\documentclass[a4paper,12pt]{article}
\usepackage[latin1]{inputenc}
\usepackage{amsthm}
\usepackage{amsfonts}
\usepackage[dvips]{graphicx}
\usepackage{latexsym}
\usepackage{amsmath}
\usepackage{color}

\newcommand{\dvec}{\mathbf{d}}
\newcommand{\vvec}{\mathbf{v}}
\newcommand{\uvec}{\mathbf{u}}
\newcommand{\wvec}{\mathbf{w}}
\newcommand{\hvec}{\mathbf{h}}
\newcommand{\gvec}{\mathbf{g}}
\newcommand{\zetavec}{\boldsymbol\zeta}
\newcommand{\V}{\textbf{V}}
\newcommand{\Gdiv}{\textbf{G}_{div}}
\newcommand{\Vdiv}{\textbf{V}_{div}}
\newcommand{\tn}[1]{\mbox {\F #1}}
\font\F=msbm10   
scaled 1000

\newtheorem{lem}{Lemma}
\newtheorem{thm}{Theorem}
\newtheorem{defn}{Definition}
\newtheorem{prop}{Proposition}
\newtheorem{cor}{Corollary}
\newtheorem{oss}{Remark}

\allowdisplaybreaks[4]

\begin{document}

\title{\large\textbf{TRAJECTORY ATTRACTORS FOR THE SUN-LIU
MODEL FOR NEMATIC LIQUID CRYSTALS IN 3D}}

\author{Sergio Frigeri and Elisabetta Rocca\thanks{The authors
are  supported by the FP7-IDEAS-ERC-StG
Grant \#256872 (EntroPhase)}\\
Dipartimento di Matematica, Universit\`{a} degli Studi di Milano\\
Via Saldini 50, 20133 Milano, Italy
\\
sergio.frigeri@unimi.it and
elisabetta.rocca@unimi.it}
\date{}
\maketitle

\begin{abstract}
\noindent In this  paper we prove the existence of a trajectory
attractor (in the sense of V.V. Chepyzhov and M.I. Vishik) for a
nonlinear PDE system coming from a 3D liquid crystal model
accounting for stretching effects. The system couples a nonlinear
evolution equation for the director $\dvec$ (introduced in order
to the describe the preferred orientation of the molecules) with
an incompressible Navier-Stokes equation for the evolution of the
velocity field $\uvec$. The technique is based on the introduction
of a suitable trajectories space and of a metric accounting for
the double-well type nonlinearity contained in the director
equation. Finally, a dissipative estimate is obtained by using a
proper integrated energy inequality.
Both the cases
of (homogeneous) Neumann and (non-homogeneous) Dirichlet boundary conditions for $\dvec$
are considered.
\end{abstract}

\tableofcontents

\section{Introduction}
\label{intro}

In this paper we prove the existence of a trajectory attractor for
the following PDE system
\begin{align}
&\uvec_t+\mbox{div}(\uvec\otimes\uvec)+\nabla p =\mbox{div}(\nu(\nabla\uvec+\nabla^T\uvec))-\mbox{div}(\nabla\dvec\odot\nabla\dvec)\nonumber\\
&-\mbox{div}(\alpha(\Delta\dvec-\nabla_{\dvec}W(\dvec))\otimes\dvec-(1-\alpha)\dvec\otimes(\Delta\dvec-\nabla_{\dvec}W(\dvec))+\hvec,\label{sy1}\\
&\dvec_t+\uvec\cdot\nabla\dvec-\alpha\dvec\cdot\nabla\uvec+(1-\alpha)\dvec\cdot\nabla^T\uvec=(\Delta\dvec-\nabla_{\dvec}W(\dvec)),\label{sy2}\\
&\mbox{div}(\uvec)=0,\label{sy3}
\end{align}
in $\Omega\times(0,T)$, where $\Omega$ is a bounded smooth domain in $\mathbb{R}^3$.

The first equation is a momentum balance ruling the evolution of the velocity field $\uvec$ ($p$ denotes the pressure of the system
and $\hvec$ is an external body force),
relation \eqref{sy3} represents the
incompressibility constraint, while \eqref{sy2} describes the dynamics of the director field $\dvec$, which represents here a vector
pointing in the preferred direction from the molecules at a neighborhood of any point of our domain.  The nonlinear function $W$ stands
for a relaxation of a constraint that should be imposed on the unitary vector $\dvec$, whose modulus should be equal to 1. In order to
relax this non-convex constraint, we introduce the double well potential $W$, which is a regular potential with some coercivity
properties (cf. next Section~\ref{preli} for the precise assumptions on $W$). For example the classical double well potential
$W(\dvec)=(|\dvec|^2-1)^2$ is included in our analysis, but also a more general  growth is admitted.
The constant $\nu$ is a positive viscosity coefficient, $\alpha \in [0,1]$ is a parameter related to the shape of the
liquid crystal molecules. For instance, the spherical, rod-like and disc-like liquid crystal molecules correspond to the cases
$\alpha= \frac{1}{2}, \, 1$ and $0$, respectively.

Concerning the notation, $\nabla_{\dvec}$ represents the gradient with respect to the variable $\dvec$.
$\nabla \dvec\odot \nabla \dvec$ denotes the $3\times 3$ matrix whose $(i,j)$-th entry is given by
$\nabla_i \dvec \cdot \nabla_j \dvec$, for $i\leq i,j\leq 3$, and $\otimes$ stands for the usual Kronecker product, i.e.,
$(\uvec\otimes\uvec)_{ij}:=\uvec_i\uvec_j$, for $i,j=1,2,3$. Finally, $\nabla^T$ indicates the transpose of the gradient.

Equations \eqref{sy1}--\eqref{sy3} come from a model introduced in \cite{SL} in order to describe the evolution of
liquid crystal substance under constant temperature (i.e. in the isothermal case). This system is obtained as a correction to
a simplification of the celebrated Leslie-Ericksen model (cf. the pioneering papers \cite{Er,Le78}) proposed in \cite{LinLiusimply}.
In \cite{LinLiusimply} the authors proposed a model in which the stretching terms
$\alpha\dvec\cdot\nabla\uvec+(1-\alpha)\dvec\cdot\nabla^T\uvec$ simply were neglected.
Coutand and Shkol\-ler \cite{CoutShkoller} proved a local well-posedness result for a model where the stretching term in equation \eqref{sy2}
was present, but, exactly due to the presence of
the stretching term, the total energy balance does not hold in that case.  To overcome such an inconvenience,
Sun and Liu proposed in \cite{SL} a variant of the Lin and Liu model \cite{LinLiusimply} in which not only the
stretching term is included in the system, but also a suitable new component is added  to the stress tensor.
Hence, the stress tensor $\tn{T}$ results as the sum of the standard stress tensor $\tn{S}= \nu(\nabla\uvec+\nabla^T\uvec)$ and
a new stretching term, i.e.
$$
\tn{T} = \tn{S}-\nabla\dvec\odot\nabla\dvec-\alpha(\Delta\dvec-\nabla_{\dvec}W(\dvec))\otimes\dvec
+(1-\alpha)\dvec\otimes(\Delta\dvec-\nabla_{\dvec}W(\dvec)\,.
$$
The resulting model \eqref{sy1}--\eqref{sy3} has been subsequently analyzed both from the point of view of
existence of strong solutions and also of their long-time behavior  in the paper \cite{wuxuliu},
where (as in \cite{SL}) the authors explicitly manifest the impossibility of proving the existence
of solutions for a standard weak formulation of the problem due to the nonlinearity of the stretching term
and of lack of maximum principle for equation \eqref{sy2}, and so of an $L^\infty$-estimate for $\dvec$.

In \cite{CR}, properly choosing the space of the test functions in the
weak momentum equation, the existence of well-defined weak solutions for the system  \eqref{sy1}--\eqref{sy3} is rigorously derived
and an integrated energy inequality is obtained. It's worth noting that the uniqueness of such solutions in the 3D case,
but also the proof of regularizing effects even in the 2D case, are not known yet, while the existence of weak solutions for
the corresponding non-isothermal system has been recently proved in \cite{FFRS}.

This results contained in the paper \cite{CR} are our starting point in order to perform the analysis on the long-time behavior of solutions.
As in \cite{CR} we consider both Neumann boundary conditions for $\dvec$ (cf., e.g.,
 \cite{LS} where it is pointed out that the Neumann boundary conditions for $\dvec$ are also suitable
for the implementation of a numerical scheme) and non-homogeneous Dirichlet ones, while for $\uvec$ we take into account only homogeneous
Dirichlet boundary conditions.
For the resulting Cauchy boundary value problem we prove
the existence of a trajectory attractor in the sense of V.V. Chepyzhov and M.I. Vishik (see \cite{CVbook, CV}).

We point out that, due to the lack of uniqueness of solutions, the choice of the notion of attractor
is essential. Indeed, there are two main approaches when one deals with dissipative systems without uniqueness
(see also \cite{CU} for a nonstandard analysis method).
The first one is based on the theory of global attractors for semigroups of multi-valued maps
(see \cite{CMR,M,MV} and also, for 3D incompressible Navier-Stokes, \cite{Ba,CF,KV,R} and references therein).
The second more geometric approach consists in working in a phase space made of trajectories
with the translation semigroup acting on them. Since the translation semigroup
is single-valued, one can then rely on the results from the classical theory of attractors
(see \cite{CV,CVbook} and also \cite{FS} and \cite{Se}).
In this paper we apply the second approach which seems more effective when the external forces are time dependent.

We essentially prove two types of results. The first one leads to a ``weaker'' definition of trajectory attractor,
but it holds true for quite general potentials $W$. The second one leads to the standard definition of trajectory attractor
in the sense of V.V. Chepyzhov and M.I. Vishik, but it holds true only for polynomially fast growing potentials $W$. In the first case,  in order to prove
the existence of the trajectory attractor under quite general assumptions on the potential $W$ (a $C^2$ function which is the sum of a convex
and ``coercive'' part and of a possibly non-convex part with Lipschitz continuous derivative), we generalize the result \cite[Thm.3.1]{CV} showing
that it is not necessary to prove the closure of the space of trajectories in the local topology
in order to obtain
the existence of the trajectory attractor.

The closure property only better characterizes the
trajectory attractor. Moreover, we subsequently prove it under more restrictive assumptions on the potential, which however are still
satisfied by the classical double-well potential $W(\dvec)=(|\dvec|^2-1)^2$. In the second case, instead, the trajectory space is defined in order to take
into account of the polynomial growth assumed on the potential $W$. In this case, we can immediately prove  the closure of
the trajectory space, leading to the standard definition of trajectory attractor in the sense of V.V. Chepyzhov and M.I. Vishik,
without any adjoint request on $W$. Let us notice that in
both cases the metric introduced on the subset of the trajectory space (suggested by the energy estimate) explicitly depends on the
potential $W$. This turns out to be meaningful in nonlinear models (cf., e.g., \cite{RS} where the phase space was explicitly depending
on the nonlinearities of the problem too).

Regarding other contributions in the literature on the long-time behavior of solutions for this system
accounting for stretching terms,
we can quote two recent
papers: \cite{GrasWu}, where the authors prove the existence of a finite-dimensional global attractor
in the 2D case and \cite{prslong} in which the authors prove - via \L ojasiewicz-Simon techniques -
the convergence of the trajectories to the stationary states under suitable conditions on the data, which are different in the
2D and 3D cases.

\paragraph{Plan of the paper.}
We split the rest of the paper in three  parts: in  Sections~\ref{sec:neu} and \ref{sec:dir} we prove, respectively, the existence of  the
trajectory attractor for \eqref{sy1}--\eqref{sy3}
 in the case of homogeneous Neumann
 and non-homogeneous Dirichlet boundary conditions for $\dvec$, finally in  the last Section~\ref{furprop} some further properties
of the trajectory attractor are studied. More specifically, in Subsection~\ref{preli}, we shall introduce some notation
and recall the main results concerning system \eqref{sy1}--\eqref{sy3},
which are proved in \cite{CR} and regarding the general theory of trajectory attractors introduced in \cite{CVbook}.
 Subsections~\ref{esitra1}, \ref{esitra2} and Section \ref{sec:dir}
 are devoted to the main results of the paper (Theorems \ref{main},
\ref{main2} and \ref{mainDir}), where the existence of the
trajectory attractor for \eqref{sy1}--\eqref{sy3}
is established under
different functional settings, assumptions on the
potential and boundary conditions for $\dvec$
(homogeneous Neumann or non-homogeneous Dirichlet).

\color{black}

\section{The case of homogeneous Neumann boundary conditions for $\dvec$}
\label{sec:neu}

In this section we deal with a suitable weak formulation of the PDE system (\ref{sy1}--\ref{sy3})
coupled with Neumann homogeneous boundary conditions for $\dvec$ and Dirichlet homogeneous ones for $\uvec$.
First, in Subsection~\ref{preli},  we introduce some notation and preliminary results which we recall for reader's convenience, then,
in Subsections~\ref{esitra1} and \ref{esitra2}, we
state and prove our main results: the existence of the trajectory attractor under two different assumptions
on the potential $W$ in \eqref{sy2}.

\subsection{Notation and preliminaries}
\label{preli}

Let us introduce the classical Hilbert spaces for the
Navier-Stokes equation
$$\Gdiv:=[\{\uvec\in C_0^{\infty}(\Omega)^3:
\mbox{div}(\uvec)=0\}]_{L^2(\Omega)^3},
$$
and
$$\Vdiv:=\{\uvec\in H_0^1(\Omega)^3:
\mbox{div}(\uvec)=0\}.
$$
We denote by $(\cdot,\cdot)$ and $\Vert\cdot\Vert$ the scalar
product and the norm, respectively, both in $L^2(\Omega)$ and in
$L^2(\Omega)^3$. We also set $\V:=H^1(\Omega)^3$ and the duality
between a Banach space $X$ and its dual $X'$ will be denoted by
$\langle\cdot,\cdot\rangle$. The space $\Vdiv$ is endowed with the
scalar product
$$(\uvec,\vvec)_{\Vdiv}:=(\nabla\uvec,\nabla\vvec),\qquad\forall\uvec,\vvec\in\Vdiv.$$

We shall also use the first eigenvalue of the Stokes operator $A$ with no-slip
boundary condition. Recall that $A:D(A)\subset \Gdiv\to\Gdiv$ is defined as
$A:=-P\Delta$  with domain $D(A)=H^2(\Omega)^3\cap\Vdiv$,
where $P:L^2(\Omega)^3\to\Gdiv$ is the Leray projector. Notice that we have
$(A\uvec,\vvec)=(\uvec,\vvec)_{\Vdiv}$, for all $\uvec\in D(A)$ and
for all $\vvec\in\Vdiv$,
and that $A^{-1}:\Gdiv\to\Gdiv$ is a self-adjoint compact operator in $\Gdiv$.
Thus, according with classical spectral theorems,
it possesses a sequence $\{\lambda_j\}$ with $0<\lambda_1\leq\lambda_2\leq\cdots$ and $\lambda_j\to\infty$,
and a family $\{\wvec_j\}\subset D(A)$ of eigenfunctions which is orthonormal in $\Gdiv$.

Let $X$ be a Banach space and $1\leq p<\infty$. We shall denote by $L^p_{tb}(0,\infty;X)$
the space of translation bounded functions in $L^p_{loc}([0,\infty);X)$. We recall that
$f\in L^p_{tb}(0,\infty;X)$ iff
$$\Vert f\Vert_{L^p_{tb}([0,\infty);X)}^p:=\sup_{t\geq 0}\int_t^{t+1}\Vert f(\tau)\Vert_X^p d\tau<\infty.$$
 Furthermore, $L^p_{loc,w}([0,\infty);X)$ will stand for the space of functions
in $L^p_{loc}([0,\infty);X)$ endowed with the local weak convergence topology, i.e., a sequence
$\{f_n\}$ converges to $f$ in $L^p_{loc,w}([0,\infty);X)$ iff
$f_n\rightharpoonup f$ weakly in $L^p(0,M;X)$, for every $M>0$.

We are ready now to recall from \cite{CR} the weak formulation of the PDE system \eqref{sy1}--\eqref{sy3} which we complement
 with the following
boundary and initial conditions
\begin{align}
&\uvec=0,\qquad\mbox{on }\Gamma\times(0,T),\label{sy4}\\
&\partial_{\boldsymbol n}\dvec=0,\qquad\mbox{on }\Gamma\times(0,T),\label{sy5}\\
&\uvec(0)=\uvec_0,\qquad\dvec(0)=\dvec_0,\qquad\mbox{in
}\Omega.\label{sy6}
\end{align}
\begin{defn}
\label{wsdefn}
A couple $\wvec=[\uvec,\dvec]$ is a weak solution to system
\eqref{sy1}--\eqref{sy5} corresponding to the initial data
$\uvec_0$, $\dvec_0$ if $\uvec$, $\dvec$ are such that
\begin{align}
&\uvec\in L_{loc}^{\infty}([0,\infty);\Gdiv)\cap L_{loc}^2([0,\infty);\Vdiv),\label{reg1}\\
&\uvec_t\in L_{loc}^2([0,\infty);W^{-1,3/2}(\Omega)^3),\label{reg2}\\
&\dvec\in L_{loc}^{\infty}([0,\infty);\V)\cap
L_{loc}^2([0,\infty);H^2(\Omega)^3),
\qquad W(\dvec)\in L_{loc}^{\infty}([0,\infty);L^1(\Omega)),\label{reg3}\\
&\dvec_t\in L_{loc}^2([0,\infty);L^{3/2}(\Omega)^3),\label{reg4}
\end{align}
$\uvec$, $\dvec$ satisfy the boundary and initial conditions
\eqref{sy5}, \eqref{sy6}, the equation \eqref{sy2} is satisfied
a.e. in $\Omega\times(0,T)$ and we have
\begin{align}
&\langle\uvec_t,\varphi\rangle-\int_{\Omega}\uvec\otimes\uvec:\nabla\varphi
+\int_{\Omega}\nu(\nabla\uvec+\nabla^T\uvec):\nabla\varphi\nonumber\\
&=\int_{\Omega}(\nabla\dvec\odot\nabla\dvec):\nabla\varphi+\alpha\int_{\Omega}(\Delta\dvec-\nabla_{\dvec}W(\dvec))\otimes\dvec:\nabla\varphi\nonumber\\
&-(1-\alpha)\int_{\Omega}\dvec\otimes(\Delta\dvec-\nabla_{\dvec}W(\dvec)):\nabla\varphi+\langle\hvec,\varphi\rangle,
\label{varform}
\end{align}
for a.e. $t>0$ and for every $\varphi\in W^{1,3}_0(\Omega)^3$ with
\upshape{div}$(\varphi)=0$.
\end{defn}

In \cite{CR} the existence of a global in time weak solution is
proved under the following assumptions on the potential $W$
\begin{align}
&W\in C^2(\mathbb{R}^3),\qquad W\geq 0,\label{Wass1}\\
&W=W_1+W_2,\quad\mbox{with } W_1\mbox{ convex and }W_2\in
C^{1,1}(\mathbb{R}^3),\label{Wass2}
\end{align}
and on the external force $\hvec$
\begin{align}
&\hvec\in L^2_{loc}([0,\infty);\Vdiv').\label{hass}
\end{align}
Namely, from \cite{CR} we recall the following
\begin{thm}
\label{existence}
Suppose that \eqref{Wass1}--\eqref{hass} are satisfied and let the
initial data be such that
\begin{align}
&\uvec_0\in\Vdiv,\qquad\dvec_0\in\V,\qquad W(\dvec_0)\in
L^1(\Omega).
\end{align}
Then, problem \eqref{sy1}--\eqref{sy3}, \eqref{sy4}--\eqref{sy6} admits a global in time
weak solution $\wvec:=[\uvec,\dvec]$ on $[0,\infty)$ corresponding
to $\uvec_0$, $\dvec_0$ and satisfying the following energy
inequality
\begin{align}
&\mathcal{E}(\wvec(t))+\int_s^t\Big(\Vert-\Delta\dvec+\nabla_{\dvec}W(\dvec)\Vert^2
+\nu\Vert\nabla\uvec\Vert^2\Big)d\tau\leq
\mathcal{E}(\wvec(s))+\int_s^t \langle\hvec,\uvec\rangle d\tau,
\label{eist}
\end{align}
for all $t\geq s$, for a.e. $s\in (0,\infty)$, including $s=0$,
where
\begin{equation}\label{en}
\mathcal{E}(\wvec):=\frac{1}{2}\Vert\uvec\Vert^2+\frac{1}{2}\Vert\nabla\dvec\Vert^2+\int_{\Omega}W(\dvec),
\qquad\wvec=[\uvec,\dvec].
\end{equation}

\end{thm}

\begin{oss}
\upshape{
 The regularity of the test function $\varphi$ can be
justified by the regularity properties of the solution which imply
that
\begin{align}
&\uvec\otimes\uvec,\quad\nabla\dvec\odot\nabla\dvec,\quad(\Delta\dvec-\nabla_{\dvec}W(\dvec))\otimes\dvec
\in L^2_{loc}([0,\infty);L^{3/2}(\Omega)^{3\times 3}).
\nonumber
\end{align}
Their divergence is therefore in $L^2_{loc}([0,\infty);W^{-1,3/2}(\Omega)^3)$.
}
\end{oss}

Let us resume some basic definitions and results from the theory of trajectory attractors
for non-autonomous evolution equations due to Chepyzhov and Vishik (see
\cite[Chap. XI and Chap. XIV]{CVbook} and \cite{CV}
for details).

Consider an abstract nonlinear non-autonomous evolution equation with
symbol $\sigma$ in a set $\Sigma$. The symbol $\sigma$ is a functional
parameter which represents all time-dependent terms (like external
forces) and coefficients of the equation.

The solutions are sought in a topological (usually Banach) space
$\mathcal{W}_M$ which consists
of vector-valued functions $w:[0,M]\to E$, where $E$ is a given Banach space.
The space $\mathcal{W}_M$
is endowed with
a given topology $\Theta_M$, such that $(\mathcal{W}_M,\Theta_M)$
is a Hausdorff topological space with a countable base. By means of $\mathcal{W}_M$
the space $\mathcal{W}_{loc}^+$ is defined as
$\mathcal{W}_{loc}^+:=\{w:[0,\infty)\to E:\Pi_{[0,M]}w\in\mathcal{W}_M,\mbox{ for all }M>0\}$, where $\Pi_{[0,M]}$ is the restriction
operator on the interval $[0,M]$.
The space $\mathcal{W}_{loc}^+$
is endowed with
a local convergence topology $\Theta_{loc}^+$,
i.e., the topology that induces the following
definition of convergence for a sequence $\{w_n\}\subset\mathcal{W}_{loc}^+$
to $w\in\mathcal{W}_{loc}^+$
$$w_n\to w\quad\mbox{in }\Theta_{loc}^+\quad\mbox{if}\quad\Pi_{[0,M]}w_n\to\Pi_{[0,M]}w\quad\mbox{in }\Theta_M,$$
for every $M>0$.
It can be seen that the space $(\mathcal{W}_{loc}^+,\Theta_{loc})$
is a Hausdorff topological space with a countable base.

For each $\sigma\in\Sigma$ let us denote by $\mathcal{K}_{\sigma}^M$ the set of {\sl some solutions}
from $\mathcal{W}_M$ and by $\mathcal{K}_{\sigma}^+$ the set of {\sl some solutions} from $\mathcal{W}_{loc}^+$.
The set $\mathcal{K}_{\sigma}^+$ is said to be a trajectory space of the evolution equation corresponding
to the symbol $\sigma\in\Sigma$.

Now, let $\mathcal{W}_b^+$ be a subspace of $\mathcal{W}_{loc}^+$
and assume that a metric $\rho_{\mathcal{W}_b^+}$ is defined on $\mathcal{W}_b^+$.
Assume also that
$\mathcal{K}_{\sigma}^+\subset\mathcal{W}_b^+$, for every $\sigma\in\Sigma$.

Recall that
the family of trajectory spaces $\{\mathcal{K}_{\sigma}^+\}_{\sigma\in\Sigma}$
is said to be \upshape{translation-coordi\-na\-ted (tr.-coord.)}
if for any $\sigma\in\Sigma$
and any $w\in\mathcal{K}_{\sigma}^+$ we have $T(t)w\in\mathcal{K}_{T(t)\sigma}^+$,
for every $t\geq 0$. The symbol space $\Sigma$ is assumed to be invariant with respect
to the translation semigroup $\{T(t)\}$, i.e., $T(t)\Sigma\subset\Sigma$, for all $t\geq 0$.

Consider the united trajectory space $\mathcal{K}_{\Sigma}^+:=\cup_{\sigma\in\Sigma}\mathcal{K}_{\sigma}^+$
of the family $\{\mathcal{K}_{\sigma}^+\}_{\sigma\in\Sigma}$.
We have $\mathcal{K}_{\Sigma}^+\subset \mathcal{W}_b^+$ and if the family $\{\mathcal{K}_{\sigma}^+\}_{\sigma\in\Sigma}$ is tr.-coord. then we have
$T(t)\mathcal{K}_{\Sigma}^+\subset\mathcal{K}_{\Sigma}^+$, for every $t\geq 0$,
i.e., the translation semigroup $\{T(t)\}$ acts on $\mathcal{K}_{\Sigma}^+$.

Introduce now the family
$$\mathcal{B}_{\Sigma}^+:=\{B\subset\mathcal{K}_{\Sigma}^+:B\mbox{ bounded in } \mathcal{W}_b^+
\mbox{ w.r.t. the metric }\rho_{\mathcal{W}_b^+}\}.$$
\begin{defn}
A set $P\subset\mathcal{W}_{loc}^+$ is said to be a uniformly (w.r.t. $\sigma\in\Sigma$)
attracting set for the family $\{\mathcal{K}_{\sigma}^+\}_{\sigma\in\Sigma}$ in the
topology $\Theta_{loc}^+$ if $P$ is uniformly (w.r.t. $\sigma\in\Sigma$)
attracting for the family $\mathcal{B}_{\Sigma}^+$, i.e. for any $B\in\mathcal{B}_{\Sigma}^+$
and for any neighbourhood $\mathcal{O}(P)$ in $\Theta_{loc}^+$
there exists $t_1\geq 0$ such that $T(t)B\subset\mathcal{O}(P)$, for every $t\geq t_1$.
\end{defn}
\begin{defn}
A set $\mathcal{A}_{\Sigma}\subset\mathcal{W}_{loc}^+$
is said to be a uniform (w.r.t. $\sigma\in\Sigma$) trajectory attractor
of the translation semigroup $\{T(t)\}$ in the topology $\Theta_{loc}^+$
if $\mathcal{A}_{\Sigma}$ is compact in $\Theta_{loc}^+$,
$\mathcal{A}_{\Sigma}$ is a uniformly (w.r.t. $\sigma\in\Sigma$)
attracting set for $\{\mathcal{K}_{\sigma}^+\}_{\sigma\in\Sigma}$ in the
topology $\Theta_{loc}^+$, and
$\mathcal{A}_{\Sigma}$ is the minimal compact and
uniformly (w.r.t. $\sigma\in\Sigma$)
attracting set for the family $\{\mathcal{K}_{\sigma}^+\}_{\sigma\in\Sigma}$ in the
topology $\Theta_{loc}^+$, i.e., if $P$ is any compact uniformly (w.r.t. $\sigma\in\Sigma$)
attracting set for the family $\{\mathcal{K}_{\sigma}^+\}_{\sigma\in\Sigma}$,
then $\mathcal{A}_{\Sigma}\subset P$.
\end{defn}
From the definition it follows that, if the trajectory attractor exists, then it is unique.

In order to prove some properties of the trajectory attractor
we need the set $\mathcal{K}_{\Sigma}^+$ to be closed
in $\Theta_{loc}^+$.  Assume that $\Sigma$ is a complete metric space.
Recall that the family $\{\mathcal{K}_{\sigma}^+\}_{\sigma\in\Sigma}$
is called $(\Theta_{loc}^+,\Sigma)-$closed if the graph set
$\cup_{\sigma\in\Sigma}\mathcal{K}_{\sigma}^+\times\{\sigma\}$
is closed in the topological space $\Theta_{loc}^+\times\Sigma$.
If $\{\mathcal{K}_{\sigma}^+\}_{\sigma\in\Sigma}$ is $(\Theta_{loc}^+,\Sigma)-$closed
and $\Sigma$ is compact, then $\mathcal{K}_{\Sigma}^+$ is closed in $\Theta_{loc}^+$.

By applying \cite[Chap. XI, Theorem 2.1]{CVbook} to the topological
space $\mathcal{W}_{loc}^+$, to the family $\mathcal{B}_{\Sigma}^+$
and to the family
$$\mathcal{B}_{\omega(\Sigma)}^+:=\{B\subset\mathcal{K}_{\omega(\Sigma)}^+:B\mbox{ bounded in } \mathcal{W}_b^+
\mbox{ w.r.t. the metric }\rho_{\mathcal{W}_b^+}\},$$
where $\mathcal{K}_{\omega(\Sigma)}^+:=\cup_{\sigma\in\omega(\Sigma)}\mathcal{K}_{\sigma}^+$
and where $\omega(\Sigma)$ is the $\omega-$limit set of $\Sigma$,
we can state
the following (see also \cite[Theorem 3.1]{CV} and \cite[Chap. XIV, Theorem 3.1]{CVbook})
\begin{thm}
\label{trajattract}
Let the spaces $(\mathcal{W}_{loc}^+,\Theta_{loc}^+)$ and $(\mathcal{W}_b^+,\rho_{\mathcal{W}_b^+})$
be as above, and the family of trajectory spaces $\{\mathcal{K}_{\sigma}^+\}_{\sigma\in\Sigma}$
corresponding to the evolution equation with symbols $\sigma\in\Sigma$ be such that
$\mathcal{K}_{\sigma}^+\subset\mathcal{W}_b^+$, for every $\sigma\in\Sigma$.
Assume there exists a subset $P\subset\mathcal{W}_{loc}^+$
which is compact in $\Theta_{loc}^+$ and uniformly (w.r.t. $\sigma\in\Sigma$)
attracting in $\Theta_{loc}^+$
for the family $\{\mathcal{K}_{\sigma}^+\}_{\sigma\in\Sigma}$ in the
topology $\Theta_{loc}^+$.
Then, the translation semigroup $\{T(t)\}_{t\geq 0}$
(acting on $\mathcal{K}_{\Sigma}^+$
if the family $\{\mathcal{K}_{\sigma}^+\}_{\sigma\in\Sigma}$ is tr.-coord.)

possesses a (unique) uniform (w.r.t. $\sigma\in\Sigma$) trajectory attractor $\mathcal{A}_{\Sigma}\subset P$.
If the semigroup $\{T(t)\}_{t\geq 0}$ is continuous in $\Theta_{loc}^+$, then $\mathcal{A}_{\Sigma}$
is strictly invariant
$$T(t)\mathcal{A}_{\Sigma}=\mathcal{A}_{\Sigma},\qquad\forall t\geq 0.$$
In addition, if the family $\{\mathcal{K}_{\sigma}^+\}_{\sigma\in\Sigma}$ is tr.-coord.
and $(\Theta_{loc}^+,\Sigma)-$closed, with $\Sigma$ a compact metric space, then
$\mathcal{A}_{\Sigma}\subset\mathcal{K}_{\Sigma}^+$ and
$$\mathcal{A}_{\Sigma}=\mathcal{A}_{\omega(\Sigma)},$$
where $\mathcal{A}_{\omega(\Sigma)}$ is the uniform (w.r.t. $\sigma\in\omega(\Sigma)$) trajectory attractor
for the family $\mathcal{B}_{\omega(\Sigma)}^+$
and $\mathcal{A}_{\omega(\Sigma)}\subset\mathcal{K}_{\omega(\Sigma)}^+$.
\end{thm}
Now, let us suppose that a dissipative estimate of the following form holds
\begin{align}
&\rho_{\mathcal{W}_b^+}(T(t)w,w_0)\leq\Lambda_0\Big(\rho_{\mathcal{W}_b^+}(w,w_0)\Big)e^{-kt}+\Lambda_1,
\qquad\forall t\geq t_0,
\label{abstdiss}
\end{align}
for every $w\in\mathcal{K}_{\Sigma}^+$, for some fixed $w_0\in\mathcal{W}_b^+$ and for
some $\Lambda_0:[0,\infty)\to[0,\infty)$ locally bounded and some constants $\Lambda_1\geq 0$, $k>0$, where
both $\Lambda_0$ and $\Lambda_1$ are independent of $w$.
Furthermore, suppose that the ball
$$B_{\mathcal{W}_b^+}(w_0,2\Lambda_1):=\{w\in\mathcal{W}_b^+:\rho_{\mathcal{W}_b^+}(w,w_0)\leq 2\Lambda_1\}$$
is compact in $\Theta_{loc}^+$.
By virtue of \eqref{abstdiss} such ball is a uniformly (w.r.t. $\sigma\in\Sigma$)
attracting set for the family $\{\mathcal{K}_{\sigma}^+\}_{\sigma\in\Sigma}$ in the
topology $\Theta_{loc}^+$ (actually, $B_{\mathcal{W}_b^+}(w_0,2\Lambda_1)$
is uniformly (w.r.t. $\sigma\in\Sigma$) absorbing for the family $\mathcal{B}_{\Sigma}^+$).
Theorem \ref{trajattract} therefore entails that
the translation semigroup $\{T(t)\}_{t\geq 0}$
possesses a (unique) uniform (w.r.t. $\sigma\in\Sigma$) trajectory attractor
$\mathcal{A}_{\Sigma}\subset B_{\mathcal{W}_b^+}(w_0,2\Lambda_1)$.

\subsection{The trajectory attractor for a general smooth potential $W$}
\label{esitra1}

We now apply the scheme described in Subsection~\ref{preli} to system \eqref{sy1}--\eqref{sy3} coupled with
boundary conditions \eqref{sy4}--\eqref{sy5} in
order to prove the existence of a trajectory attractor for that system.

For $M>0$ introduce the space
\begin{align}
\mathcal{W}_M:=&\Big\{[\uvec,\dvec]\in L^{\infty}(0,M;\Gdiv\times\V)\cap
L^2(0,M;\Vdiv\times H^2(\Omega)^3):\nonumber\\
&
\uvec_t\in L^2(0,M;W^{-1,3/2}(\Omega)^3),\dvec_t\in L^2(0,M;L^{3/2}(\Omega)^3)
\Big\},
\label{FM}
\end{align}
endowed with the weak topology $\Theta_M$ which induces the following notion
of weak convergence: a sequence $\{[\uvec_m,\dvec_m]\}\subset\mathcal{W}_M$
is said to converge to $[\uvec,\dvec]\in\mathcal{W}_M$ in $\Theta_M$
if
\begin{align}
&\uvec_m\rightharpoonup\uvec\qquad\mbox{weakly}^{\ast}\mbox{ in }L^{\infty}(0,M;\Gdiv)
\mbox{ and weakly in }L^2(0,M;\Vdiv),\label{wconv1}\\
&(\uvec_m)_t\rightharpoonup\uvec_t\qquad\mbox{weakly in }L^2(0,M;W^{-1,3/2}(\Omega)^3),\\
&\dvec_m\rightharpoonup\dvec\qquad\mbox{weakly}^{\ast}\mbox{ in }L^{\infty}(0,M;\V)
\mbox{ and weakly in }L^2(0,M;H^2(\Omega)^3),\label{conv3}\\
&(\dvec_m)_t\rightharpoonup\dvec_t\qquad\mbox{weakly in }L^2(0,M;L^{3/2}(\Omega)^3).\label{wconv4}
\end{align}
Then the space
\begin{align}
\mathcal{W}_{loc}^+:=&\Big\{[\uvec,\dvec]\in L_{loc}^{\infty}([0,\infty);\Gdiv\times\V)\cap
L_{loc}^2([0,\infty);\Vdiv\times H^2(\Omega)^3):\nonumber\\
&
\uvec_t\in L_{loc}^2([0,\infty);W^{-1,3/2}(\Omega)^3),\dvec_t\in L_{loc}^2([0,\infty);L^{3/2}(\Omega)^3)
\Big\}
\label{Floc}
\end{align}
is defined, as well as the inductive limit weak topology $\Theta_{loc}^+$.

In $\mathcal{W}_{loc}^+$ we consider the following
subspace
\begin{align}
\mathcal{W}_b^+:=&\Big\{[\uvec,\dvec]\in L^{\infty}(0,\infty;\Gdiv\times\V)\cap
L_{tb}^2(0,\infty;\Vdiv\times H^2(\Omega)^3):\nonumber\\
&
\uvec_t\in L_{tb}^2(0,\infty;W^{-1,3/2}(\Omega)^3),\dvec_t\in L_{tb}^2(0,\infty;L^{3/2}(\Omega)^3),
W(\dvec)\in L^{\infty}(0,\infty;L^1(\Omega))
\Big\},
\label{Fb}
\end{align}
and on $\mathcal{W}_b^+$ we define the following metric
\begin{align}
&\rho_{\mathcal{W}_{b}^+}(\wvec_1,\wvec_2):=\Vert\wvec_1-\wvec_2\Vert_{L^{\infty}(0,\infty;\Gdiv\times\V)}
+\Vert\wvec_1-\wvec_2\Vert_{L_{tb}^2(0,\infty;\Vdiv\times H^2(\Omega)^3)}\\
&+\Vert(\uvec_1)_t-(\uvec_2)_t\Vert_{L_{tb}^2(0,\infty;W^{-1,3/2}(\Omega)^3)}
+\Vert(\dvec_1)_t-(\dvec_2)_t\Vert_{L_{tb}^2(0,\infty;L^{3/2}(\Omega)^3)}\\
&+\Big\Vert\int_{\Omega}W(\dvec_1)-\int_{\Omega}W(\dvec_2)\Big\Vert_{L^{\infty}(0,\infty)}^{1/2},
\label{dFb}
\end{align}
for every $\wvec_1=[\uvec_1,\dvec_1],\wvec_2=[\uvec_2,\dvec_2]\in\mathcal{W}_b^+$.

\begin{defn}
\label{trajdef}
For every $\hvec\in L^2_{loc}([0,\infty);\Vdiv')$ the trajectory space $\mathcal{K}_{\hvec}^+$
of system \eqref{sy1}--\eqref{sy3}, \eqref{sy4}--\eqref{sy5} with external force $\hvec$
is the set of all weak solutions $\wvec=[\uvec,\dvec]$
of this system with the regularity properties \eqref{reg1}--\eqref{reg4} for $\uvec$, $\dvec$,
and satisfying the energy inequality \eqref{eist}
for all $t\geq s$ and for a.a. $s\in(0,\infty)$.
\end{defn}

The trajectory space $\mathcal{K}_{\hvec}^M$ on the bounded interval $[0,M]$
can be defined similarly, for every $M>0$.

\begin{oss}
{\upshape
Notice that in the definition of the trajectory space $\mathcal{K}_{\hvec}^+$ we do not assume
that the energy inequality \eqref{eist} is satisfied also for $s=0$.
In this way the family $\{\mathcal{K}_{\hvec}^+\}_{\hvec\in\Sigma}$
($\Sigma$ may be a generic symbol space included in $L^2_{loc}([0,\infty);\Vdiv')$)
is tr.-coord. and therefore the translation semigroup $\{T(t)\}$ acts on $\mathcal{K}_{\Sigma}^+$.
}
\end{oss}

According to Theorem \ref{existence}, if \eqref{Wass1} and \eqref{Wass2} hold, then, for every $\wvec_0=[\uvec_0,\dvec_0]$ such that
\begin{align}
&\uvec_0\in\Vdiv,\qquad\dvec_0\in\V,\qquad W(\dvec_0)\in L^1(\Omega),
\label{idatass}
\end{align}
and every $\hvec$ such that
\begin{align}
\hvec\in L^2_{loc}([0,\infty);\Vdiv')
\label{hdatass}
\end{align}
there exists a trajectory $\wvec\in\mathcal{K}_{\hvec}^+$
for which $\wvec(0)=\wvec_0$.

Consider now
$$\hvec_0\in L_{tb}^2(0,\infty;\Vdiv')$$
so that $\hvec_0$
is translation compact (tr.-c.) in $L_{loc,w}^2([0,\infty);\Vdiv')$
(see, e.g., \cite[Proposition 6.8]{CV}). For the symbol space
$\Sigma$ we choose the hull of $\hvec_0$ in $L_{loc,w}^2([0,\infty);\Vdiv')$
\begin{align}
\Sigma=\mathcal{H}_+(\hvec_0):=[\{T(t)\hvec_0,t\geq 0\}]_{L_{loc,w}^2([0,\infty);\Vdiv')}
\label{hull}
\end{align}
which is a compact metric space.
Recall (see \cite[Proposition 6.9]{CV}) that every $\hvec\in\mathcal{H}_+(\hvec_0)$
is also tr.-c. in $L_{loc,w}^2([0,\infty);\Vdiv')$ and
\begin{align}
&\Vert\hvec\Vert_{L_{tb}^2(0,\infty;\Vdiv')}\leq\Vert\hvec_0\Vert_{L_{tb}^2(0,\infty;\Vdiv')},
\qquad\forall\hvec\in\mathcal{H}_+(\hvec_0).
\label{hh0}
\end{align}

In order to prove the closure of the space of the trajectory attractor,
we shall also assume that $\hvec_0$ is tr.-c. in $L_{loc}^2([0,\infty);\Vdiv')$
or tr.-c. in $L_{loc,w}^2([0,\infty);\Gdiv)$.
The latter condition is equivalent to the assumption that $\hvec_0\in L_{tb}^2(0,\infty;\Gdiv)$.
It is not difficult to prove that the hull of $\hvec_0$ with one of these assumptions
(defined as in \eqref{hull} with the clousure in the above spaces)
coincides with the hull $\mathcal{H}_+(\hvec_0)$ defined as in \eqref{hull}.
\color{black}

 In order to state our first result on the existence of the trajectory attractor, we  shall make the following assumption
on the potential $W$
\begin{description}
\item[(W1)]
$W$ satisfies \eqref{Wass1}, \eqref{Wass2} and there exist
$c_0\geq 0$, $c_1>0$, $c_2\in\mathbb{R}$ and $\delta>0$ such that
\begin{align}
&W_1(\dvec)\leq c_0(1+|\nabla_{\dvec} W_1(\dvec)|^2),\label{w1}\\
&W_1(\dvec)\geq c_1|\dvec|^{2+\delta}-c_2,\label{w2}
\end{align}
for every $\dvec\in\mathbb{R}^3$.
\end{description}

Let us now state the following Lemma which will be useful in order to prove our next
main Theorem~\ref{main}.
\begin{lem}
Take assumption {\bf (W1)} on $W$, then
there exist $\kappa,\,\eta\,, {\color{red}l}>0$ (independent of $\dvec$)
such that
\begin{align}
&\Vert-\Delta\dvec+\nabla_{\dvec} W(\dvec)\Vert^2\geq\kappa\Vert\nabla\dvec\Vert^2
+\eta\int_{\Omega}W(\dvec)-{\color{red}l},
\label{prelimest}
\end{align}
for all $\dvec\in H^2(\Omega)^3$, with $\partial_{\boldsymbol{n}}\dvec=0$ on $\partial\Omega$.
\end{lem}
\begin{proof}
Using {\bf (W1)}, we have
\begin{align}
&\Vert-\Delta\dvec+\nabla_{\dvec} W_1(\dvec)\Vert^2=\Vert-\Delta\dvec+\dvec\Vert^2+\Vert\dvec\Vert^2
-2(\dvec,-\Delta\dvec+\dvec)\nonumber\\
&+\Vert\nabla_{\dvec} W_1(\dvec)\Vert^2+2(-\Delta\dvec,\nabla_{\dvec} W_1(\dvec)).
\label{wd}
\end{align}
By means of \eqref{w1} we obtain
\begin{align}
&\Vert\nabla_{\dvec} W_1(\dvec)\Vert^2\geq\frac{1}{c_0}\int_{\Omega}W_1(\dvec)-|\Omega|\,.
\end{align}
Moreover,  using the convexity of $W_1$ (which implies that $(-\Delta\dvec,\nabla_{\dvec} W_1(\dvec))\geq 0$)
and the fact that
\eqref{w2} implies that $W_1(\dvec)\geq c_3|\dvec|^2-c_4$, from \eqref{wd} we get
\begin{align}
&\Vert-\Delta\dvec+\nabla_{\dvec} W_1(\dvec)\Vert^2\geq\epsilon\Vert-\Delta\dvec+\dvec\Vert^2
-\frac{\epsilon}{1-\epsilon}\Vert\dvec\Vert^2+\frac{1}{c_0}\int_{\Omega}W_1(\dvec)-|\Omega|\nonumber\\
&\geq\epsilon c_i\Vert\nabla\dvec\Vert^2+\epsilon\Big(c_i-\frac{1}{1-\epsilon}\Big)\Vert\dvec\Vert^2+
\frac{1}{c_0}\int_{\Omega}W_1(\dvec)-|\Omega|\nonumber\\
&\geq\epsilon c_i\Vert\nabla\dvec\Vert^2+\Big(\frac{c_3}{2c_0}-\epsilon\Big(c_i-\frac{1}{1-\epsilon}\Big)\Big)
\Vert\dvec\Vert^2+
\frac{1}{2c_0}\int_{\Omega}W_1(\dvec)-c_5\nonumber\\
&\geq\epsilon c_i\Vert\nabla\dvec\Vert^2+\frac{1}{2c_0}\int_{\Omega}W_1(\dvec)-c_5,
\label{la}
\end{align}
provided $\epsilon$ is chosen small enough. In \eqref{la} the positive constant
$c_i$ is such that
$$c_i\Vert\dvec\Vert_{\V}\leq\Vert-\Delta\dvec+\dvec\Vert,$$
for every $\dvec\in H^2(\Omega)^3$, with $\partial_{\boldsymbol{n}}\dvec=0$ on $\partial\Omega$.
Now, from \eqref{la} we have
\begin{align}
&\Vert-\Delta\dvec+\nabla_{\dvec} W(\dvec)\Vert^2\geq\frac{1}{2}\Vert-\Delta\dvec
+\nabla_{\dvec} W_1(\dvec)\Vert^2
-\Vert\nabla_{\dvec} W_2(\dvec)\Vert^2
\nonumber\\
&\geq\frac{\epsilon c_i}{2}\Vert\nabla\dvec\Vert^2+\frac{1}{4c_0}\int_{\Omega}W_1(\dvec)
-\Vert\nabla_{\dvec} W_2(\dvec)\Vert^2-\frac{c_5}{2},
\end{align}
and observe that, due to \eqref{w2} and to the assumption \eqref{Wass2} on $W_2$, we can choose $\eta>0$
such that
$$\frac{1}{4c_0}W_1(\dvec)-|\nabla_{\dvec} W_2(\dvec)|^2\geq\eta W(\dvec)-c_{\eta},\qquad\forall\dvec\in\mathbb{R}^3.$$
We therefore get \eqref{prelimest}
with $\kappa=\epsilon c_i/2$ and $l$ depending on $\Omega$, $W$
and with $\eta$ depending on $W$ only.
\end{proof}

In order to prove that the united trajectory space
$\mathcal{K}_{\mathcal{H}_+(\hvec_0)}^+$ is closed in $\Theta_{loc}^+$
we shall also need
the following  growth assumption
on $W$
\begin{description}
\item[(W2)]There exists $b>0$ such that
$$W(\dvec)\leq b(1+|\dvec|^6),\qquad\forall\dvec\in\mathbb{R}^3.$$
\end{description}

\begin{oss}
\upshape{Notice that both assumptions {\bf (W1)} and {\bf (W2)} are satisfied
in the case of the physically interesting double-well
potential
$$W(\dvec)=(|\dvec|^2-1)^2.$$
This function is usually assumed as a good smooth approximation
for a potential penalizing the deviation of the length $|\dvec|$
from the value 1, which is due to liquid crystal
molecules being of similar size.
}
\end{oss}

We can now state our first  main result
\begin{thm}
\label{main}
Let {\bf (W1)} holds and that $\hvec_0\in L_{tb}^2(0,\infty;\Vdiv')$.
Then,
the semigroup
$\{T(t)\}$ acting on $\mathcal{K}_{\mathcal{H}_+(\hvec_0)}^+$
possesses the uniform (w.r.t. $\hvec\in\mathcal{H}_+(\hvec_0)$)
trajectory attractor $\mathcal{A}_{\mathcal{H}_+(\hvec_0)}$.
This set is strictly invariant, bounded in $\mathcal{W}_b^+$ and compact in $\Theta_{loc}^+$.
In addition, if {\bf (W2)} holds
and $\hvec_0$ is tr.-c. in $L_{loc}^2([0,\infty);\Vdiv')$
or $\hvec_0\in L_{tb}^2(0,\infty;\Gdiv)$,
\color{black}
then $\mathcal{K}_{\mathcal{H}_+(\hvec_0)}^+$
is closed in $\Theta_{loc}^+$,
$\mathcal{A}_{\mathcal{H}_+(\hvec_0)}\subset\mathcal{K}_{\mathcal{H}_+(\hvec_0)}^+$
and
$$\mathcal{A}_{\mathcal{H}_+(\hvec_0)}=\mathcal{A}_{\omega(\mathcal{H}_+(\hvec_0))}.$$
\end{thm}

For the proof of Theorem \ref{main} we need two propositions.
The first proposition establishes a dissipative estimate of the form \eqref{abstdiss}
for our problem

\begin{prop}
Assume {\bf (W1)} holds and that $\hvec_0\in L_{tb}^2(0,\infty;\Vdiv')$. Then, for all $\hvec\in\mathcal{H}_+(\hvec_0)$
we have $\mathcal{K}_{\hvec}^+\subset\mathcal{W}_b^+$ and the following dissipative estimate holds
\begin{align}
& \rho_{\mathcal{W}_{b}^+}(T(t)\wvec,\boldsymbol{0})
\leq c \rho_{\mathcal{W}_{b}^+}^2(\wvec,\boldsymbol{0})e^{-\frac{k}{2}t}
+\Lambda_0,\qquad\forall t\geq 1,
\label{dissipativeest}
\end{align}
for all $\wvec\in\mathcal{K}_{\hvec}^+$. Here $\Lambda_0$, $k$ and $c$ are positive constants
(independent of $\wvec$) that depend on $W$, $\Omega$, $\nu$ with only
$\Lambda_0$ depending on the norm of $\hvec_0$ in $L_{tb}^2(0,\infty;\Vdiv')$.
In particular $k$ can be given by $k=\min(\eta,2\kappa,\nu\lambda_1)$,
where $\lambda_1$ is the first eigenvalue of the Stokes operator and $\eta,\kappa$
are such that \eqref{prelimest} holds.
\label{dissprop}
\end{prop}

\begin{proof}

Take now $\wvec\in\mathcal{K}_{\hvec}^+$, with $\hvec\in\mathcal{H}_+(\hvec_0)$.
Recalling the definition of the energy $\mathcal{E}$ \eqref{en}, using \eqref{prelimest} and Poincar\'{e} inequality
we have
\begin{align}
&\Vert-\Delta\dvec+\nabla_{\dvec} W(\dvec)\Vert^2+\frac{\nu}{2}\Vert\nabla\uvec\Vert^2
\geq k\mathcal{E}(\wvec)-l,\qquad\wvec=[\uvec,\dvec]
\label{ka}
\end{align}
where $k=\min(\eta,2\kappa,\nu\lambda_1)$, $\lambda_1$ being the first eigenvalue
of the Stokes operator,
and $l$ (depending on $\Omega$, $W$ only) is the same as in \eqref{prelimest}.

Therefore, by combining \eqref{ka} with the energy inequality \eqref{eist}
we deduce that $\wvec$ satisfies
the integral inequality
\begin{align}
\mathcal{E}(\wvec(t))+k\int_0^t\mathcal{E}(\wvec(\tau))d\tau
\leq &l(t-s)+\frac{1}{2\nu}\int_s^t\Vert\hvec(\tau)\Vert_{\Vdiv'}^2d\tau
+\mathcal{E}(\wvec(s))+k\int_0^s\mathcal{E}(\wvec(\tau))d\tau,
\end{align}
for all $t\geq s$ and for a.e. $s\in(0,\infty)$.
We can now apply a suitable modification
\cite[Lemma 1]{FG} of an integral Gronwall lemma due to Ball
\cite[Lemma 7.2]{Ba}
and deduce that
\begin{align}
&\mathcal{E}(\wvec(t))\leq\mathcal{E}(\wvec(s))e^{-k(t-s)}
+\frac{1}{2\nu}\int_s^t e^{-k(t-\tau)}\Big(\Vert\hvec(\tau)\Vert^2_{\Vdiv'}+2\nu l\Big)d\tau,
\label{n0}
\end{align}
for all $t\geq s$ and for a.e. $s\in(0,\infty)$.
Notice that, due to the regularity properties
of the solution, which imply
that $\uvec\in C_w([0,\infty);\Gdiv)$,
$\dvec\in C_w([0,\infty);\V)$
(and hence $\dvec\in C([0,\infty);L^2(\Omega)^3)$),
  and to the fact that, thanks to \eqref{Wass2},
$W$ is a quadratic perturbation of a convex function,
then $\mathcal{E}(\wvec(\cdot)):[0,\infty)\to\mathbb{R}$ is lower semicontinuous.

Hence
\begin{align}
&\mathcal{E}(\wvec(t))\leq e^k\sup_{s\in(0,1)}\mathcal{E}(\wvec(s))e^{-kt}
+\frac{1}{2\nu}\int_0^t e^{-k(t-\tau)}\Big(\Vert\hvec(\tau)\Vert^2_{\Vdiv'}+2\nu l\Big)d\tau\nonumber\\
&\leq e^k\sup_{s\in(0,1)}\mathcal{E}(\wvec(s))e^{-kt}+K^2,
\qquad\forall t\geq 1,
\label{n1}
\end{align}
where
$$K^2=\frac{l}{k}+\frac{1}{2\nu(1-e^{-k})}\Vert\hvec_0\Vert_{L^2_{tb}(0,\infty;\Vdiv')}^2.$$
Now, observe that due to \eqref{w2} we have
\begin{align}
&\mathcal{E}(\wvec)\geq c_6\Big(\Vert\uvec\Vert^2+\Vert\dvec\Vert_{\V}^2+\int_{\Omega}W(\dvec)\Big)-c_7,
\label{n2}
\end{align}
and
\begin{align}
&\sup_{s\in(0,1)}\mathcal{E}(\wvec(s))\leq\frac{1}{2}\Vert \uvec\Vert_{L^{\infty}(0,1;\Gdiv)}^2
+\frac{1}{2}\Vert\nabla\dvec\Vert_{L^{\infty}(0,1;L^2(\Omega)^{3\times 3})}^2
+\sup_{s\in(0,1)}\int_{\Omega}W(\dvec(s))\nonumber\\
&\leq c \rho_{\mathcal{W}_{b}^+}^2(\wvec,\boldsymbol{0}),\qquad\forall\wvec=[\uvec,\dvec]\in\mathcal{W}_{b}^+.
\label{n3}
\end{align}
 Henceforth in this proof we shall denote by $c$ a nonnegative
constant, which may vary even within the same line,
that possibly depends on $W$, $\Omega$ and $\nu$, but is independent of
$\wvec$ and $\hvec_0$.

By combining \eqref{n2} and \eqref{n3} with \eqref{n1} we get
\begin{align}
\Vert\uvec(t)\Vert+\Vert\dvec(t)\Vert_{\V}+\Big(\int_{\Omega}W(\dvec(t))\Big)^{1/2}
\leq c \rho_{\mathcal{W}_{b}^+}(\wvec,\boldsymbol{0})e^{-\frac{k}{2}t}
+cK+c,\qquad\forall t\geq 1,
\label{n5}
\end{align}
and hence
\begin{align}
&\Vert T(t)\uvec\Vert_{L^{\infty}(0,\infty;\Gdiv)}
+\Vert T(t)\dvec\Vert_{L^{\infty}(0,\infty;\V)}
+\Big\Vert\int_{\Omega}W(T(t)\dvec)\Big\Vert_{L^{\infty}(0,\infty)}^{1/2}\nonumber\\
&\leq c \rho_{\mathcal{W}_{b}^+}(\wvec,\boldsymbol{0})e^{-\frac{k}{2}t}
+cK+c,\qquad\forall t\geq 1.
\label{n10}
\end{align}
From the energy inequality \eqref{eist} we have
\begin{align}
&\int_t^{t+1}\Big(\Vert-\Delta\dvec+\nabla_{\dvec}W(\dvec)\Vert^2+\frac{\nu}{2}\Vert\nabla\uvec\Vert^2\Big)d\tau
\nonumber\\
&\leq\mathcal{E}(\wvec(t))-\mathcal{E}(\wvec(t+1))
+\frac{1}{2\nu}\int_t^{t+1}\Vert\hvec(\tau)\Vert_{\Vdiv'}^2d\tau,
\label{n4}
\end{align}
for a.e. $t>0$.

Notice that, thanks to the convexity of $W_1$ and to the assumption \eqref{Wass2} on $W_2$, we have
\begin{align}
&\Vert-\Delta\dvec+\nabla_{\dvec}W(\dvec)\Vert^2\geq\frac{1}{2}\Vert-\Delta\dvec+\nabla_{\dvec}W_1(\dvec)\Vert^2
-\Vert\nabla_{\dvec}W_2(\dvec)\Vert^2
\nonumber\\
&\geq\frac{1}{4}\Vert-\Delta\dvec+\dvec\Vert^2-\frac{1}{2}\Vert\dvec\Vert^2
-\Vert\nabla_{\dvec}W_2(\dvec)\Vert^2
\nonumber\\
&\geq\frac{1}{4}\Vert-\Delta\dvec+\dvec\Vert^2-c\Vert\dvec\Vert^2-c,
\end{align}
and therefore \eqref{n4} and \eqref{n5} entail
\begin{align}
&\int_t^{t+1}\Big(\frac{1}{4}\Vert\dvec(\tau)\Vert_{H^2(\Omega)^3}^2+\frac{\nu}{2}\Vert\nabla\uvec(\tau)\Vert^2
\Big)d\tau
\leq c \rho_{\mathcal{W}_{b}^+}^2(\wvec,\boldsymbol{0})e^{-kt}
+cK^2+c,\qquad\forall t\geq 1,
\label{n6}
\end{align}
which implies that
\begin{align}
&\Vert T(t)\dvec\Vert_{L^2_{tb}(0,\infty;H^2(\Omega)^3)}+
\Vert T(t)\uvec\Vert_{L^2_{tb}(0,\infty;\Vdiv)}
\leq c \rho_{\mathcal{W}_{b}^+}(\wvec,\boldsymbol{0})e^{-\frac{k}{2}t}
+cK+c,\qquad\forall t\geq 1.
\label{n11}
\end{align}
Now, recall that, due to the interpolation inequality
\[L^{\infty}(0,M;L^2(\Omega))\cap L^2(0,M;L^6(\Omega))\subset L^4(0,M;L^3(\Omega))\]
and to the regularity property of the solution, we have that
$$\uvec\cdot\nabla\dvec,\quad\dvec\cdot\nabla\uvec\in L_{loc}^2([0,\infty);L^{3/2}(\Omega)^3),$$
and so
\begin{align}
&\Vert\uvec\cdot\nabla\dvec\Vert_{L^2(t,t+1;L^{3/2}(\Omega)^3)}
\leq\Vert\uvec\Vert_{L^4(t,t+1;L^3(\Omega)^3)}\Vert\nabla\dvec\Vert_{L^4(t,t+1;L^3(\Omega)^{3\times 3})}
\nonumber\\
&\leq c(\Vert\uvec\Vert_{L^{\infty}(t,t+1;\Gdiv)}+\Vert\uvec\Vert_{L^2(t,t+1;\Vdiv)})
(\Vert\dvec\Vert_{L^{\infty}(t,t+1;\V)}
+\Vert\dvec\Vert_{L^2(t,t+1;H^2(\Omega)^3)}),
\nonumber
\end{align}
and
\begin{align}
&\Vert\dvec\cdot\nabla\uvec\Vert_{L^2(t,t+1;L^{3/2}(\Omega)^3)}
\leq c\Vert\dvec\Vert_{L^{\infty}(t,t+1;\V)}\Vert\uvec\Vert_{L^2(t,t+1;\Vdiv)}.
\nonumber
\end{align}
By using \eqref{n5} and \eqref{n6} we hence get
\begin{align}
&\Vert\uvec\cdot\nabla\dvec\Vert_{L^2(t,t+1;L^{3/2}(\Omega)^3)}
+\Vert\dvec\cdot\nabla\uvec\Vert_{L^2(t,t+1;L^{3/2}(\Omega)^3)}\leq
c \rho_{\mathcal{W}_{b}^+}^2(\wvec,\boldsymbol{0})e^{-kt}
+cK^2+c,
\label{n7}
\end{align}
for all $t\geq 1$. Furthermore, from \eqref{n4} and \eqref{n5} we have
\begin{align}
&\Vert-\Delta\dvec+\nabla W(\dvec)\Vert_{L^2(t,t+1;L^2(\Omega)^3)}
\leq c \rho_{\mathcal{W}_{b}^+}(\wvec,\boldsymbol{0})e^{-\frac{k}{2}t}
+cK+c,\qquad\forall t\geq 1.
\label{n8}
\end{align}
Therefore, by using \eqref{sy2}, \eqref{n7} and \eqref{n8} we obtain
\begin{align}
&\Vert\dvec_t\Vert_{L^2(t,t+1;L^{3/2}(\Omega)^3)}
\leq c \rho_{\mathcal{W}_{b}^+}^2(\wvec,\boldsymbol{0})e^{-\frac{k}{2}t}
+cK^2+c,\qquad\forall t\geq 1,
\nonumber
\end{align}
and from this last inequality
\begin{align}
&\Vert T(t)\dvec_t\Vert_{L^2_{tb}(0,\infty;L^{3/2}(\Omega)^3)}
\leq c \rho_{\mathcal{W}_{b}^+}^2(\wvec,\boldsymbol{0})e^{-\frac{k}{2}t}
+cK^2+c,\qquad\forall t\geq 1.
\label{n12}
\end{align}

Finally, observe that the regularity properties of the solution also entail
\begin{align}
&\uvec\otimes\uvec,\quad\nabla\dvec\odot\nabla\dvec,\quad(\Delta\dvec-\nabla_{\dvec}W(\dvec))\otimes\dvec
\in L^2_{loc}([0,\infty);L^{3/2}(\Omega)^{3\times 3}),
\label{consreg}
\end{align}
and we have
\begin{align}
&\Vert\uvec\otimes\uvec\Vert_{L^2(t,t+1;L^{3/2}(\Omega)^{3\times 3})}
\leq\Vert\uvec\Vert_{L^4(t,t+1;L^3(\Omega)^3)}^2
\leq c(\Vert\uvec\Vert_{L^{\infty}(t,t+1;\Gdiv)}+\Vert\uvec\Vert_{L^2(t,t+1;\Vdiv)})^2,
\end{align}
\begin{align}\nonumber
\Vert\nabla\dvec\odot\nabla\dvec\Vert_{L^2(t,t+1;L^{3/2}(\Omega)^{3\times 3})}
&\leq\Vert\nabla\dvec\Vert_{L^4(t,t+1;L^3(\Omega)^3)}^2\\
\nonumber
&
\leq c(\Vert\dvec\Vert_{L^{\infty}(t,t+1;\V)}
+\Vert\dvec\Vert_{L^2(t,t+1;H^2(\Omega)^3)})^2,
\end{align}
and
\begin{align}\nonumber
\Vert(\Delta\dvec-\nabla_{\dvec}W(\dvec))&\otimes\dvec\Vert_{L^2(t,t+1;L^{3/2}(\Omega)^{3\times 3})}\\
\nonumber
&\leq\Vert(\Delta\dvec-\nabla_{\dvec}W(\dvec))\Vert_{L^2(t,t+1;L^2(\Omega)^3)}
\Vert\dvec\Vert_{L^{\infty}(t,t+1;\V)}.
\end{align}
Therefore, from the variational formulation \eqref{varform} for the equation of the velocity we get
\begin{align}
&\Vert\uvec_t\Vert_{L^2(t,t+1;W^{-1,3/2}(\Omega)^3)}
\leq\Vert\uvec\otimes\uvec\Vert_{L^2(t,t+1;L^{3/2}(\Omega)^{3\times 3})}+
c\mu\Vert\uvec\Vert_{L^2(t,t+1;\Vdiv)}\nonumber\\
&+\Vert\nabla\dvec\odot\nabla\dvec\Vert_{L^2(t,t+1;L^{3/2}(\Omega)^{3\times 3})}
+\Vert(\Delta\dvec-\nabla_{\dvec}W(\dvec))\otimes\dvec\Vert_{L^2(t,t+1;L^{3/2}(\Omega)^{3\times 3})}
\nonumber\\
&+\Vert\dvec\otimes(\Delta\dvec-\nabla_{\dvec}W(\dvec))\Vert_{L^2(t,t+1;L^{3/2}(\Omega)^{3\times 3})}
+c\Vert\hvec\Vert_{L^2(t,t+1;\Vdiv')}.
\label{n9}
\end{align}
By combining \eqref{n9} with the previous estimates and with \eqref{n5}, \eqref{n6} and with \eqref{n8}, we easily
obtain
\begin{align}
&\Vert\uvec_t\Vert_{L^2(t,t+1;W^{-1,3/2}(\Omega)^3)}\leq
c \rho_{\mathcal{W}_{b}^+}^2(\wvec,\boldsymbol{0})e^{-\frac{k}{2}t}
+cK^2+c,\qquad\forall t\geq 1,
\nonumber
\end{align}
whence
\begin{align}
&\Vert T(t)\uvec_t\Vert_{L^2_{tb}(0,\infty;W^{-1,3/2}(\Omega)^3)}\leq
c \rho_{\mathcal{W}_{b}^+}^2(\wvec,\boldsymbol{0})e^{-\frac{k}{2}t}
+cK^2+c,\qquad\forall t\geq 1.
\label{n13}
\end{align}
Collecting now \eqref{n10}, \eqref{n11}, \eqref{n12} and \eqref{n13} we deduce that
$\mathcal{K}_{\hvec}^+\subset\mathcal{W}_b^+$ and that
\eqref{dissipativeest} holds
with $\Lambda_0=cK^2+c$.
\end{proof}

The next proposition states that
$\{\mathcal{K}_{\hvec}^M\}_{\hvec\in L^2(0,M;\Vdiv')}$
is $(\Theta_M,L^2(0,M;\Vdiv'))-$closed
and $\{\mathcal{K}_{\hvec}^M\}_{\hvec\in L^2(0,M;\Gdiv)}$
is $(\Theta_M,L^2_w(0,M;\Gdiv))-$closed,
for every $M>0$,
under the further assumption {\bf (W2)} on $W$.

\begin{prop}
\label{closprop}
Suppose that {\bf (W1)} and {\bf (W2)} hold. Let
$\wvec_m:=[\uvec_m,\dvec_m]\in\mathcal{K}_{\hvec_m}^M$ be
such that $\{\wvec_m\}$ converges to $\wvec:=[\uvec,\dvec]$ in $\Theta_M$ and
$\{\hvec_m\}$ be such that one of the following convergence assumptions holds
\begin{description}
\item[(a)] $\hvec_m\in L^2(0,M;\Vdiv')$ and $\hvec_m\to\hvec$, strongly in $ L^2(0,M;\Vdiv')$,
\item[(b)] $\hvec_m\in L^2(0,M;\Gdiv)$ and $\hvec_m\rightharpoonup\hvec$, weakly in $ L^2(0,M;\Gdiv)$.
\end{description}
Then $\wvec\in\mathcal{K}_{\hvec}^M$.
\end{prop}
\color{black}
\begin{proof}
Since $\wvec_m=[\uvec_m,\dvec_m]\in\mathcal{K}_{\hvec_m}^M$, then, every weak solution $\wvec_m$
is such that: (i) the regularity properties \eqref{reg1}-\eqref{reg4} hold for
each solution $\wvec_m$, (ii)
the weak formulation \eqref{varform} for $\uvec_m$ corresponding to the
external force $\hvec_m$
and \eqref{sy2}, \eqref{sy5} for $\dvec_m$
are satisfied, and (iii) the energy inequality
\begin{align}
&\mathcal{E}(\wvec_m(t))+\int_s^t\Big(\Vert-\Delta\dvec_m+\nabla_{\dvec}W(\dvec_m)\Vert^2
+\nu\Vert\nabla\uvec_m\Vert^2\Big)d\tau
\leq \mathcal{E}(\wvec_m(s))+\int_s^t \langle\hvec_m,\uvec_m\rangle d\tau
\label{eim}
\end{align}
holds for every $m\in\mathbb{N}$, for all $t$ and a.e. $s$ with $t\geq s$ and $s,t\in[0,M]$.
The weak convergences \eqref{wconv1}-\eqref{wconv4} imply that the sequence $\{\uvec_m\}$ is bounded
in $L^{\infty}(0,M;\Gdiv)$, the sequence $\{\dvec_m\}$ is bounded in $L^{\infty}(0,M;\V)$ and hence
also in $ L^{\infty}(0,M;L^6(\Omega)^3)$. The growth assumption {\bf (W2)} then entails
$$|\mathcal{E}(\wvec_m(s))|\leq c,$$
for every $m$ and a.e. $s\in[0,M]$.
Therefore, by using \eqref{eim} and the convergence assumption for the sequence $\{\hvec_m\}$ we deduce
that
\begin{align}
\Vert-\Delta\dvec_m+\nabla_{d}W(\dvec_m)\Vert_{L^2(0,M;L^2(\Omega)^3)}\leq c.
\label{H2cont}
\end{align}
Since, by \eqref{conv3} the sequence $\{-\Delta\dvec_m\}$ is bounded in $L^2(0,M;L^2(\Omega)^3)$,
then we infer that $\{\nabla_{d}W(\dvec_m)\}$ is bounded in $L^2(0,M;L^2(\Omega)^3)$ as well
and therefore, up to a subsequence, $\nabla_{d}W(\dvec_m)\rightharpoonup\boldsymbol{G}$ weakly in $L^2(0,M;L^2(\Omega)^3)$.
Since we also have, as a consequence of the convergences \eqref{conv3}, \eqref{wconv4}
and of Aubin-Lions lemma, that $\dvec_m\to\dvec$ strongly in $L^2(0,M;\V)$,
we deduce that $\boldsymbol G=\nabla_{d}W(\dvec)$.

It is easy to check that $\wvec=[\uvec,\dvec]$ is a weak solution corresponding to the external force $\hvec$.
Indeed, we can take $\varphi\in\mathcal{D}(\Omega)^3$ with div$\varphi=0$, write the variational formulation
\eqref{varform} for $\uvec_m$ and equation \eqref{sy2} for $\dvec_m$
and pass to the limit as $m\to\infty$.
Then we can use the weak convergences \eqref{wconv1}-\eqref{wconv4} which imply the strong
convergences $\uvec_m\to\uvec$ in $L^2(0,M;\Gdiv)$ (and hence $\uvec_m\otimes\uvec_m\to\uvec\otimes\uvec$
in $L^1(0,M;L^1(\Omega)^{3\times 3})$), $\nabla\dvec_m\to\nabla\dvec$ in $L^2(0,M;L^2(\Omega)^{3\times 3})$
and also the assumed convergence for the sequence
$\{\hvec_m\}$ to conclude that $\uvec$ satisfies the variational formulation \eqref{varform}
with external force $\hvec$ for every test function
$\varphi\in\mathcal{D}(\Omega)^3$ with div$\varphi=0$, and that $\dvec$ satisfies \eqref{sy2}.
By density and \eqref{consreg}, the weak
formulation for $\uvec$ is satisfied also for every
$\varphi\in W^{1,3}_0(\Omega)$ with div$\varphi=0$.

It remains to prove that $\wvec$ satisfies the energy inequality \eqref{eist} with external force
$\hvec$ on $[0,M]$.  Let us first assume the convergence
condition (a) for the sequence $\{\hvec_m\}$.
\color{black}
We then consider \eqref{eim}, pass to the limit as $m\to\infty$,
use the strong and weak convergences for the sequences $\{\uvec_m\}$, $\{\dvec_m\}$ and,
on the left hand side of the inequality, the lower semicontinuity of the $L^2(0,M;L^2(\Omega))-$
norm and Fatou's lemma for the nonlinear integral term. On the right hand side we use the fact that,
since, by Aubin-Lions lemma, we have the compact and continuous embeddings
\begin{align}
&L^2(0,M;H^2(\Omega)^3)\cap H^1(0,M;L^{3/2}(\Omega)^3)\hookrightarrow\hookrightarrow
L^2(0,M;H^{2-\delta}(\Omega)^3)\hookrightarrow L^2(0,M;C(\overline{\Omega})^3),
\label{AubLio}
\end{align}
for $0<\delta<1/2$, then, up to a subsequence, we have $\dvec_m(s)\to\dvec(s)$ in $C(\overline{\Omega})^3$
for a.e. $s\in[0,M]$ and therefore
$$\int_{\Omega}W(\dvec_m(s))\to\int_{\Omega}W(\dvec(s)),\qquad\mbox{a.e. }s\in[0,M].$$
We hence recover \eqref{eist} for $\wvec$ with forcing term $\hvec$.

On the other hand, if (b) holds, we can argue as in \cite[Chap. XV, Prop. 1.1]{CVbook}
and exploit the strong convergence $\uvec_m\to\uvec$ in $L^2(0,M;\Gdiv)$ which implies
$$\int_s^t\langle \hvec_m(\tau),\uvec_m(\tau)\rangle d\tau\to\int_s^t \langle \hvec(\tau),\uvec(\tau)\rangle d\tau,
\qquad\mbox{as }m\to\infty.$$

In both cases (a) and (b)
we therefore conclude that $\wvec\in\mathcal{K}_{\hvec}^M$.
\color{black}

\end{proof}

\begin{proof}[Proof of Theorem \ref{main}]
By Proposition \ref{dissprop} the ball
$$B_{\mathcal{W}_b^+}(\boldsymbol 0,2\Lambda_0)
:=\{\wvec\in\mathcal{W}_b^+:\rho_{\mathcal{W}_b^+}(\wvec,\boldsymbol 0)\leq 2\Lambda_0\}$$
is a uniformly (w.r.t. $\hvec\in\mathcal{H}_+(\hvec_0)$)
absorbing set for the family $\{\mathcal{K}_{\hvec}^+\}_{\hvec\in\mathcal{H}_+(\hvec_0)}$.
Such ball is also precompact in $\Theta_{loc}^+$.
The first part of Theorem \ref{trajattract} entails the existence of
the uniform (w.r.t. $\hvec\in\mathcal{H}_+(\hvec_0)$)
trajectory attractor $\mathcal{A}_{\mathcal{H}_+(\hvec_0)}
\subset B_{\mathcal{W}_b^+}(\boldsymbol 0,2\Lambda_0)$. This
set is compact in $\Theta_{loc}^+$ and,
since $T(t)$ is obviously continuous in $\Theta_{loc}^+$,
$\mathcal{A}_{\mathcal{H}_+(\hvec_0)}$ is also strictly invariant.

Furthermore, assuming also {\bf (W2)} and that $\hvec_0$ is tr.-c. in $L^2_{loc}([0,\infty);\Vdiv')$ or
$\hvec_0\in L^2_{tb}(0,\infty;\Gdiv)$, then
\color{black}
Proposition \ref{closprop} implies that
$\{\mathcal{K}_{\hvec}^+\}_{\hvec\in\mathcal{H}_+(\hvec_0)}$
is $(\Theta_{loc}^+,\mathcal{H}_+(\hvec_0))-$closed
when {\bf (W2)} holds true.
Since $\mathcal{H}_+(\hvec_0)$ is a compact metric space,
then $\mathcal{K}_{\mathcal{H}_+(\hvec_0)}^+$ is closed in $\Theta_{loc}^+$
and the second part of Theorem \ref{trajattract}
allows to conclude the proof.
\end{proof}


\subsection{The trajectory attractor for a polynomial potential $W$}
\label{esitra2}

The results on the existence of the trajectory attractor and on its closure
property can be recovered under alternative functional setting and assumptions
on the potential.
Indeed, let $p\geq 2$ and for every $M>0$ introduce the space
\begin{align}
\mathcal{W}_{p,M}:=&\Big\{[\uvec,\dvec]\in L^{\infty}(0,M;\Gdiv\times(\V\cap L^p(\Omega)^3))\cap
L^2(0,M;\Vdiv\times H^2(\Omega)^3):\nonumber\\
&
\uvec_t\in L^2(0,M;W^{-1,3/2}(\Omega)^3),\dvec_t\in L^2(0,M;L^{3/2}(\Omega)^3)
\Big\}.
\label{FMp}
\end{align}
The topology $\Theta_{p,M}$ on $\mathcal{W}_{p,M}$ is now chosen to induce the following notion
of weak convergence:
a sequence $\{[\uvec_m,\dvec_m]\}\subset\mathcal{W}_{p,M}$
is said to converge to $[\uvec,\dvec]\in\mathcal{W}_M$ in $\Theta_{p,M}$
if \eqref{wconv1}--\eqref{wconv4} hold and if in addition
\begin{align}
&\dvec_m\rightharpoonup\dvec\qquad\mbox{weakly}^{\ast}\mbox{ in }L^{\infty}(0,M;L^p(\Omega)^3).
\label{conv5}
\end{align}
Then define
\begin{align}
\mathcal{W}_{p,loc}^+:=&\Big\{[\uvec,\dvec]\in L_{loc}^{\infty}([0,\infty);\Gdiv\times(\V\cap L^p(\Omega)^3))\cap
L_{loc}^2([0,\infty);\Vdiv\times H^2(\Omega)^3):\nonumber\\
&
\uvec_t\in L_{loc}^2([0,\infty);W^{-1,3/2}(\Omega)^3),\dvec_t\in L_{loc}^2([0,\infty);L^{3/2}(\Omega)^3)
\Big\},
\label{Fploc}
\end{align}
endowed with its inductive limit topology $\Theta_{p,loc}^+$, and
\begin{align}
\mathcal{W}_{p,b}^+:=&\Big\{[\uvec,\dvec]\in L^{\infty}(0,\infty;\Gdiv\times(\V\cap L^p(\Omega)^3))\cap
L_{tb}^2(0,\infty;\Vdiv\times H^2(\Omega)^3):\nonumber\\
&
\uvec_t\in L_{tb}^2(0,\infty;W^{-1,3/2}(\Omega)^3),\dvec_t\in L_{tb}^2(0,\infty;L^{3/2}(\Omega)^3)\Big\},
\label{Fpb}
\end{align}
which is now a Banach space with the norm
\begin{align}
&\Vert\wvec_1-\wvec_2\Vert_{\mathcal{W}_{p,b}^+}:=
\Vert\wvec_1-\wvec_2\Vert_{L^{\infty}(0,\infty;\Gdiv\times(\V\cap L^p(\Omega)^3))}
+\Vert\wvec_1-\wvec_2\Vert_{L_{tb}^2(0,\infty;\Vdiv\times H^2(\Omega)^3)}\\
&+\Vert(\uvec_1)_t-(\uvec_2)_t\Vert_{L_{tb}^2(0,\infty;W^{-1,3/2}(\Omega)^3)}
+\Vert(\dvec_1)_t-(\dvec_2)_t\Vert_{L_{tb}^2(0,\infty;L^{3/2}(\Omega)^3)},
\label{Fpbnorm}
\end{align}
for every $\wvec_1=[\uvec_1,\dvec_1],\wvec_2=[\uvec_2,\dvec_2]\in\mathcal{W}_{p,b}^+$.
On the potential $W$ we now need the following assumption
\begin{description}
\item[(W3)]There exist two positive constants $C_1$, $C_2$ and $p\in (2,+\infty)$ such that
\begin{align}
C_1(|\dvec|^p-1)\leq W(\dvec)\leq C_2(1+|\dvec|^p),\qquad\forall\dvec\in\mathbb{R}^3,
\label{altassW}
\end{align}
\end{description}
\begin{oss}{\upshape
Let us note that assumption {\bf (W3)} is satisfied with $p=4$ in the standard double-well potential case
$W(\dvec)=(|\dvec|^2-1)^2$.}
\end{oss}

For every $\hvec\in L^2_{loc}([0,\infty);\Vdiv')$ the trajectory space
$\mathcal{K}_{p,\hvec}^+$ can be defined exactly as in Definition \ref{trajdef}
with the additional requirement that
\begin{align}
\dvec\in L^{\infty}_{loc}([0,\infty);L^p(\Omega)^3).
\end{align}
Thanks to {\bf (W3)}, then, if assumptions \eqref{Wass1}, \eqref{Wass2}
are satisfied, Theorem \ref{existence} ensures that for every $\wvec_0=[\uvec_0,\dvec_0]$
such that $\uvec_0\in\Vdiv$, $\dvec_0\in\V\cap L^p(\Omega)^3$
and every $\hvec\in L^2_{loc}([0,\infty);\Vdiv')$ there exists a trajectory
$\wvec\in\mathcal{K}_{p,\hvec}^+$ such that $\wvec(0)=\wvec_0$.
Furthermore, the space $\mathcal{K}_{p,\hvec}^M$ of trajectories on the interval $[0,M]$
can be defined in an obvious way, as well as the united trajectory space
$$\mathcal{K}_{p,\mathcal{H}_+(\hvec_0)}^+:=\bigcup_{\hvec\in\mathcal{H}_+(\hvec_0)}\mathcal{K}_{p,\hvec}^+.$$

In place of {\bf (W1)} on the potential $W$ we shall therefore make the following assumption
\begin{description}
\item[{\bf (W1)}$^\ast$]
W satisfies \eqref{Wass1}, \eqref{Wass2},
and \eqref{w1}.
\end{description}

Instead of Theorem \ref{main} we can now prove the following
\begin{thm}
\label{main2}
Assume that {\bf (W1)}$^\ast$ and {\bf (W3)} hold
and that $\hvec_0\in L^2_{tb}(0,\infty;\Vdiv')$.
Then, $\{T(t)\}$ acting on $\mathcal{K}_{p,\mathcal{H}_+(\hvec_0)}^+$ possesses the uniform
(w.r.t. $\hvec\in\mathcal{H}_+(\hvec_0)$) trajectory attractor
$\mathcal{A}_{p,\mathcal{H}_+(\hvec_0)}$.
This set is strictly invariant, bounded in $\mathcal{W}_{p,b}^+$, compact in
$\Theta_{p,loc}^+$.
In addition, if $\hvec_0$ is tr.-c. in $L^2_{loc}([0,\infty);\Vdiv')$ or
$\hvec_0\in L^2_{tb}(0,\infty;\Gdiv)$, then
$\mathcal{K}_{p,\mathcal{H}_+(\hvec_0)}^+$ is closed in
$\Theta_{p,loc}^+$,
$\mathcal{A}_{p,\mathcal{H}_+(\hvec_0)}
\subset\mathcal{K}_{p,\mathcal{H}_+(\hvec_0)}^+$ and
$$\mathcal{A}_{p,\mathcal{H}_+(\hvec_0)}=\mathcal{A}_{p,\omega(\mathcal{H}_+(\hvec_0))}.$$
\color{black}
\end{thm}

Similarly to Theorem \ref{main}, Theorem \ref{main2} is a consequence of two propositions. The first one
concerns with a dissipative estimate, and the second one establishes the closure property of the space of trajectories.

\begin{prop}
Let {\bf (W1)}$^{\ast}$ and {\bf (W3)} be satisfied and assume that $\hvec_0\in L^2_{tb}(0,\infty;\Vdiv')$.
Then, for every $\hvec\in\mathcal{H}_+(\hvec_0)$ we have $\mathcal{K}_{p,\hvec}^+\subset\mathcal{W}_{p,b}^+$
and
\begin{align}
\Vert T(t)\wvec\Vert_{\mathcal{W}_{p,b}^+}\leq \Gamma(\Vert\wvec\Vert_{\mathcal{W}_{p,b}^+})e^{-\sigma t}
+\Gamma_0,\qquad\forall t\geq 1,
\label{diss2}
\end{align}
for every $\wvec\in\mathcal{K}_{p,\hvec}^+$. Here,
$\Gamma_0$, $\sigma$ and $\Gamma$ are two positive constants and a
monotone positive increasing function, respectively, (independent
of $\wvec$) that depend on $W$, $\Omega$, $\nu$ and $p$, with only
$\Gamma_0$ depending on the $L^2_{tb}(0,\infty;\Vdiv')-$norm of
$\hvec_0$. \label{dissipativeest2}
\end{prop}

\begin{proof}
The proof is analogous to the one of Proposition \ref{dissprop} with some modifications.
Indeed, it is easy to check that estimate \eqref{prelimest} still holds and also that \eqref{ka}--\eqref{n1}
can be rewritten. On the other hand, \eqref{n2} and \eqref{n3} will now be replaced by
\begin{align}
&\mathcal{E}(\wvec)\geq c_6\Big(\Vert\uvec\Vert^2+\Vert\dvec\Vert_{\V}^2+\Vert\dvec\Vert_{L^p(\Omega)^3}^p\Big)-c_7,
\end{align}
\begin{align}
&\sup_{s\in(0,1)}\mathcal{E}(\wvec(s))\leq\frac{1}{2}\Vert \uvec\Vert_{L^{\infty}(0,1;\Gdiv)}^2
+\frac{1}{2}\Vert\nabla\dvec\Vert_{L^{\infty}(0,1;L^2(\Omega)^{3\times 3})}^2
+c_8\Vert\dvec\Vert_{L^{\infty}(0,1;L^p(\Omega)^3)}^p+c_9,
\end{align}
respectively. Here all nonnegative constants $c_i$ depend possibly on $W$, $\Omega$, $\nu$ and
$p$, but do not depend
neither on the solution $\wvec$, nor on $\hvec_0$.
Hence, in place of \eqref{n5} we get
\begin{align}
&\Vert\uvec\Vert^2+\Vert\dvec\Vert_{\V}^2+\Vert\dvec\Vert_{L^p(\Omega)^3}^p\nonumber\\
&
\leq c_{10}\Big(\Vert \uvec\Vert_{L^{\infty}(0,1;\Gdiv)}^2
+\Vert\nabla\dvec\Vert_{L^{\infty}(0,1;L^2(\Omega)^{3\times 3})}^2
+\Vert\dvec\Vert_{L^{\infty}(0,1;L^p(\Omega)^3)}^p\Big) e^{-kt}+K^2+c_{11},\,\forall t\geq 1,
\end{align}
the constant $K$ being given as in the proof of Proposition \ref{dissprop}.
Hence we have
\begin{align}
&\Vert T(t)\uvec\Vert_{L^{\infty}(0,\infty;\Gdiv)}+\Vert T(t)\dvec\Vert_{L^{\infty}(0,\infty;\V)}
+\Vert T(t)\dvec\Vert_{L^{\infty}(0,\infty;L^p(\Omega)^3)}\nonumber\\
&\leq c_{12}\Big(\Vert \uvec\Vert_{L^{\infty}(0,1;\Gdiv)}
+\Vert\nabla\dvec\Vert_{L^{\infty}(0,1;L^2(\Omega)^{3\times 3})}
+\Vert\dvec\Vert_{L^{\infty}(0,1;L^p(\Omega)^3)}^{p/2}\Big)e^{-\frac{k}{p}t}+K+c_{13},\quad\forall t\geq 1.
\label{mod1}
\end{align}
Once we have \eqref{mod1}, for the remaining part of the proof we can argue as in
the proof of Proposition \ref{dissprop}. At the end we arrive at \eqref{diss2}
with $\sigma=k/p$, $\Gamma(R)=c_{14}R^p$ and $\Gamma_0=c_{15}K^2+c_{16}$.
\end{proof}

\begin{prop}
Assume that {\bf (W1)}$^{\ast}$ and {\bf (W3)} are satisfied.
Let $\wvec_m:=[\uvec_m,\dvec_m]\in\mathcal{K}_{p,\hvec_m}^M$ be such that
$\{\wvec_m\}$ converges to $\wvec:=[\uvec,\dvec]$ in $\Theta_{p,M}$
and assume that $\{\hvec_m\}$ and $\hvec$ satisfy (a) or (b)
from Proposition \ref{closprop}.
\color{black}
Then $\wvec\in\mathcal{K}_{p,\hvec}^M$.
\label{closprop2}
\end{prop}

\begin{proof}
The argument is the same as in the proof of Proposition \ref{closprop}.
The only difference is that now, once we write \eqref{eim}, the control
$|\mathcal{E}(\wvec_m(s))|\leq c$ for all $m$ and for a.e. $s\in(0,M)$
is ensured by the weak$^{\ast}$ convergence \eqref{conv5}.
\end{proof}


\section{The case of non-homogeneous Dirichlet boundary conditions for $\dvec$}
\label{sec:dir}

Let us now consider the physically relevant case of non-homogeneous Dirichlet boundary condition
for $\dvec$
\begin{align}
&\dvec|_{\Gamma}=\gvec,
\label{nonhomDir}
\end{align}
where the boundary datum $\gvec$ is supposed to be at least such that
$$\gvec\in H^1_{loc}([0,\infty);H^{-1/2}(\Gamma)^3)
\cap L^2_{loc}([0,\infty);H^{3/2}(\Gamma)^3)
\,.$$
This condition, together with \eqref{idatass}, \eqref{hdatass}
and with the compatibility condition
$$\dvec_0|_{\Gamma}=\gvec(0),$$
\color{black}
ensure the existence of a global
in time weak solution on $[0,\infty)$ corresponding to $\uvec_0$, $\dvec_0$ and $\gvec$, $\hvec$
with the regularity properties \eqref{reg1}--\eqref{reg4} and
satisfying the following energy inequality (cf. \cite{CR})
\begin{align}
&\mathcal{E}(\wvec(t))+\int_s^t\Big(\Vert-\Delta\dvec+\nabla_{\dvec}W(\dvec)\Vert^2
+\nu\Vert\nabla\uvec\Vert^2\Big)d\tau\nonumber\\
&\leq
\mathcal{E}(\wvec(s))+
\int_s^t\langle\gvec_t,\partial_{\boldsymbol n}\dvec\rangle_{H^{-1/2}(\Gamma)^3\times H^{1/2}(\Gamma)^3} d\tau
+\int_s^t \langle\hvec,\uvec\rangle d\tau,\qquad\wvec:=[\uvec,\dvec]
\label{eig}
\end{align}
for all $t\geq s$, for a.e. $s\in (0,\infty)$, including $s=0$,
where the energy functional $\mathcal{E}$ is the same as for the case of homogeneous Neumann
boundary condition for $\dvec$ (cf. Theorem \ref{existence}).

We can recover the result on the existence of the trajectory attractor also for the case
of non-homogeneous boundary condition for $\dvec$, assuming that either {\bf (W1)} or {\bf (W1)}$^*$ and {\bf (W3)} holds
for the potential $W$.
Indeed, if {\bf (W1)} holds, introducing the spaces $\mathcal{W}_M$, $\mathcal{W}_{loc}^+$ and $\mathcal{W}_b^+$
defined as in \eqref{FM}, \eqref{Floc} and \eqref{Fb}, respectively,
we can define the trajectory space $\mathcal{K}_{\gvec,\hvec}^+$ of system \eqref{sy1}--\eqref{sy4}, \eqref{nonhomDir} with
boundary datum $\gvec$ and external force $\hvec$
as the set of all weak solutions $\wvec=[\uvec,\dvec]$ on the time interval $[0,\infty)$ to this system
satisfying the
energy inequality \eqref{eig} for all $t\geq s$ and for a.a. $s\in(0,\infty)$.
On the other hand, if {\bf (W1)}$^*$  and {\bf (W3)} hold, the spaces $\mathcal{W}_{p,M}$, $\mathcal{W}_{p,loc}^+$ and $\mathcal{W}_{p,b}^+$,
defined as in \eqref{FMp}, \eqref{Fploc} and \eqref{Fpb}, respectively,
can be introduced and the trajectory space $\mathcal{K}_{p,\gvec,\hvec}^+$
can be defined in a similar way. The definition of the trajectory spaces
$\mathcal{K}_{\gvec,\hvec}^M$ and $\mathcal{K}_{p,\gvec,\hvec}^M$ on the bounded time interval $[0,M]$
is obvious.

Let us now introduce the symbol spaces for the Dirichlet datum $\gvec$
$$\Xi_M:=\{\gvec\in C([0,M];H^{3/2}(\Gamma)^3):\gvec_t\in L^2(0,M;H^{-1/2}(\Gamma)^3)\},$$
$$\Xi^+_{loc}:=\{\gvec\in C([0,\infty);H^{3/2}(\Gamma)^3):\gvec_t\in L^2_{loc}([0,\infty);H^{-1/2}(\Gamma)^3)\},$$
and also
\begin{align*}
&\Xi_{loc,w}^+:=\{\gvec\in C([0,\infty);H^{3/2}(\Gamma)^3):\gvec_t\in L^2_{loc,w}([0,\infty);H^{-1/2}(\Gamma)^3)\}.
\end{align*}
Take then
\begin{align}
&\gvec_0\mbox{ tr.-c. in }C([0,\infty);H^{3/2}(\Gamma)^3),\quad\partial_t\gvec_0
\in L^2_{tb}(0,\infty;H^{-1/2}(\Gamma)^3),
\label{g0ass}
\end{align}
so that $\gvec_0$ is tr.-c. in $\Xi_{loc,w}^+$ and set
\begin{align}
&\mathcal{H}_+(\gvec_0):=[\{T(t)\gvec_0,t\geq 0\}]_{\Xi_{loc,w}^+}.
\label{hullg}
\end{align}
We shall also assume that $\gvec_0$ is tr.-c. in $\Xi_{loc}^+$.
It can be proved that in this case the hull of $\gvec_0$
(defined as in \eqref{hullg} with the closure in $\Xi_{loc}^+$)
coincides with the hull $\mathcal{H}_+(\gvec_0)$ defined
as in \eqref{hull}.
\color{black}
The united trajectory spaces are now given by
$$\mathcal{K}^+_{\mathcal{H}_+(\gvec_0)\times\mathcal{H}_+(\hvec_0)}
=\bigcup_{\gvec\in\mathcal{H}_+(\gvec_0),\hvec\in\mathcal{H}_+(\hvec_0)}\mathcal{K}_{\gvec,\hvec}^+,$$
or by
$$\mathcal{K}^+_{p,\mathcal{H}_+(\gvec_0)\times\mathcal{H}_+(\hvec_0)}
=\bigcup_{\gvec\in\mathcal{H}_+(\gvec_0),\hvec\in\mathcal{H}_+(\hvec_0)}\mathcal{K}_{p,\gvec,\hvec}^+.$$
\color{black}

We therefore can state the following

\begin{thm}
\label{mainDir}
Assume that {\bf (W1)} ({\bf (W1)}$^*$ and {\bf (W3)}) holds and that
$\gvec_0$ is tr.-c. in $C([0,\infty);H^{3/2}(\Gamma)^3)$
with $\partial_t\gvec_0
\in L^2_{tb}(0,\infty;H^{-1/2}(\Gamma)^3)$, and
$\hvec_0\in L^2_{tb}(0,\infty;\Vdiv')$.
Then, $\{T(t)\}$ acting on $\mathcal{K}^+_{\mathcal{H}_+(\gvec_0)\times\mathcal{H}_+(\hvec_0)}$
($\mathcal{K}^+_{p,\mathcal{H}_+(\gvec_0)\times\mathcal{H}_+(\hvec_0)}$)
possesses the uniform (w.r.t. $[\gvec,\hvec]\in\mathcal{H}_+(\gvec_0)\times\mathcal{H}_+(\hvec_0)$)
trajectory attractor $\mathcal{A}_{\mathcal{H}_+(\gvec_0)\times\mathcal{H}_+(\hvec_0)}$
($\mathcal{A}_{p,\mathcal{H}_+(\gvec_0)\times\mathcal{H}_+(\hvec_0)}$).
This set is strictly invariant, bounded in $\mathcal{W}_b^+$ ($\mathcal{W}_{p,b}^+$) and compact in $\Theta_{loc}^+$ ($\Theta_{p,loc}^+$).
In addition, if {\bf (W2)} holds (or if {\bf (W1)}$^*$ and {\bf (W3)} hold),
if $\gvec_0$ is tr.-c. in $\Xi_{loc}^+$ and
if $\hvec_0$ is tr.-c. in $L^2_{loc}([0,\infty);\Vdiv')$
or $\hvec_0\in L^2_{tb}(0,\infty;\Gdiv)$,
\color{black}
then
$\mathcal{K}^+_{\mathcal{H}_+(\gvec_0)\times\mathcal{H}_+(\hvec_0)}$
($\mathcal{K}^+_{p,\mathcal{H}_+(\gvec_0)\times\mathcal{H}_+(\hvec_0)}$)
is closed in $\Theta_{loc}^+$ ($\Theta_{p,loc}^+$),
\color{black}
$\mathcal{A}_{\mathcal{H}_+(\gvec_0)\times\mathcal{H}_+(\hvec_0)}
\subset\mathcal{K}^+_{\mathcal{H}_+(\gvec_0)\times\mathcal{H}_+(\hvec_0)}$
($\mathcal{A}_{p,\mathcal{H}_+(\gvec_0)\times\mathcal{H}_+(\hvec_0)}
\subset\mathcal{K}^+_{p,\mathcal{H}_+(\gvec_0)\times\mathcal{H}_+(\hvec_0)}$)
 and
$$\mathcal{A}_{\mathcal{H}_+(\gvec_0)\times\mathcal{H}_+(\hvec_0)}
=\mathcal{A}_{\omega(\mathcal{H}_+(\gvec_0)\times\mathcal{H}_+(\hvec_0))}\quad
(\mathcal{A}_{p,\mathcal{H}_+(\gvec_0)\times\mathcal{H}_+(\hvec_0)}
=\mathcal{A}_{p,\omega(\mathcal{H}_+(\gvec_0)\times\mathcal{H}_+(\hvec_0))}).$$
\end{thm}

\begin{proof}
We can easily recover
a dissipative estimate of the form \eqref{dissipativeest} (or \eqref{diss2}) and also the closure property
of the space of trajectories.
Let us start by proving the dissipative inequality, taking, e.g., assumption {\bf (W1)}.
First observe that, due to the convexity of $W_1$ and to \eqref{w1}, we have
\begin{align}
&\Vert-\Delta\dvec+\nabla_{\dvec}W(\dvec)\Vert^2\geq\frac{1}{2}\Vert\Delta\dvec\Vert^2+
\frac{1}{2c_0}\int_{\Omega}W_1(\dvec)-\frac{1}{2}|\Omega|\nonumber\\
&
-\int_{\Gamma}\partial_{\boldsymbol n}\dvec\cdot\nabla_{\dvec}W_1(\gvec)-\Vert\nabla_{\dvec}W_2(\dvec)\Vert^2.
\color{black}
\label{p1}
\end{align}
Observe that, since $W_1\in C^{1,1}_{loc}(\mathbb{R}^3)$ and $\dvec\in H^2(\Omega)$,
then it is not hard to prove that the trace of $\nabla_{\dvec}W_1(\dvec)$
on $\Gamma$ is $\nabla_{\dvec}W_1(\gvec)$.
\color{black}
By using trace and $H^2-$elliptic regularity estimates we can write
\begin{align}
&|\langle\gvec_t,\partial_{\boldsymbol{n}}\dvec\rangle|_{H^{-1/2}(\Gamma)^3\times H^{1/2}(\Gamma)^3}\leq
\Vert\gvec_t\Vert_{H^{-1/2}(\Gamma)^3}\Vert\partial_{\boldsymbol{n}}\dvec\Vert_{ H^{1/2}(\Gamma)^3}\nonumber\\
&\leq a_1\Vert\gvec_t\Vert_{H^{-1/2}(\Gamma)^3}\Vert\dvec\Vert_{ H^2(\Omega)^3}
\leq a_2\Vert\gvec_t\Vert_{H^{-1/2}(\Gamma)^3}\Big(\Vert\Delta\dvec\Vert+\Vert\gvec\Vert_{H^{3/2}(\Gamma)^3}\Big)
\nonumber\\
&\leq\frac{1}{8}\color{black}\Vert\Delta\dvec\Vert^2+2a_2^2\Vert\gvec_t\Vert_{H^{-1/2}(\Gamma)^3}^2
+\Vert\gvec\Vert_{H^{3/2}(\Gamma)^3}^2,
\label{p2}
\end{align}
and the boundary integral term in \eqref{p1} can be estimated as
\begin{align}
&|\langle \partial_{\boldsymbol n}\dvec,\nabla_{\dvec}W_1(\gvec)\rangle|
\leq a_3(\Vert\Delta\dvec\Vert+\Vert\gvec\Vert_{H^{3/2}(\Gamma)^3})\Vert\nabla_{\dvec}W_1(\gvec)\Vert_{L^2(\Gamma)^3}\nonumber\\
&\leq\frac{1}{8}\Vert\Delta\dvec\Vert^2+a_4\Vert\gvec\Vert_{H^{3/2}(\Gamma)^3}^2
+a_4\Vert\nabla_{\dvec}W_1(\gvec)\Vert_{L^2(\Gamma)^3}^2.
\label{p3}
\end{align}
\color{black}
Henceforth we shall denote by $a_i$ some nonnegative constants
which depend only on $\Omega$ (like $a_i$ for
$i=1,\cdots, 6$) or on $\Omega$ and $W$.\color{black}\\
Take now $w\in\mathcal{K}_{\gvec,\hvec}^+$ with $\gvec\in\mathcal{H}_+(\gvec_0)$
and $\hvec\in\mathcal{H}_+(\hvec_0)$. Then, plugging \eqref{p1}--\eqref{p3}
into \eqref{eig} and using also the elliptic estimate
$a_5\Vert\dvec\Vert_{\V}^2\leq\Vert\Delta\dvec\Vert^2+\Vert\gvec\Vert_{H^{1/2}(\Gamma)^3}^2$
and Poincar\'{e} inequality for $\uvec$, we obtain
\begin{align}
&\mathcal{E}(\wvec(t))+\int_s^t\Big\{ \frac{a_5}{4}\Vert\nabla\dvec\Vert^2+\frac{1}{2c_0}\int_{\Omega}W_1(\dvec)+\frac{a_5}{4}\Vert\dvec\Vert^2
-\frac{|\Omega|}{2}-\Vert\nabla_{\dvec}W_2(\dvec)\Vert^2+\frac{\nu\lambda_1}{2}\Vert\uvec\Vert^2\Big\}d\tau
\nonumber\\
\nonumber
&\leq\mathcal{E}(\wvec(s))+\int_s^t\Big\{2a_2^2\Vert\gvec_t\Vert_{H^{-1/2}(\Gamma)^3}^2
+a_6\Vert\gvec\Vert_{H^{3/2}(\Gamma)^3}^2\\
\nonumber
&\qquad\qquad\qquad\qquad+a_4\Vert\nabla_{\dvec}W_1(\gvec)\Vert_{L^2(\Gamma)^3}^2
\color{black}+
\frac{1}{2\nu}\Vert\hvec\Vert_{\Vdiv'}^2\Big\}d\tau,\nonumber
\end{align}
for all $t\geq s$ and for a.a. $s\in(0,\infty)$.
On account of \eqref{w2} and of the assumption on $W_2$ we can now choose $a_7$ and $a_8$ (depending
on $\Omega$ and also on $W$) such that
$$\frac{1}{2c_0}W_1(\dvec)+\frac{a_5}{4}|\dvec|^2-\frac{1}{2}-|\nabla_{\dvec}W_2(\dvec)|^2
\geq a_7 W(\dvec)-a_8.$$
We are thus led to the following inequality
\begin{align}
&\mathcal{E}(\wvec(t))+k'\int_0^t\mathcal{E}(\wvec(\tau))d\tau\leq l'(t-s)
+\int_s^t m(\tau)d\tau+\mathcal{E}(\wvec(s))+k'\int_0^s\mathcal{E}(\wvec(\tau))d\tau,
\label{inindir}
\end{align}
for all $t\geq s$ and for a.a. $s\in(0,\infty)$, where
$k'=\min(a_5/2,a_7,\lambda_1\nu)$, $l'=a_8|\Omega|$
\color{black}
and
\begin{align}
&m(t):=2a_2^2\Vert\gvec_t(t)\Vert_{H^{-1/2}(\Gamma)^3}^2
+a_6\Vert\gvec(t)\Vert_{H^{3/2}(\Gamma)^3}^2+a_4\Vert\nabla_{\dvec}W_1(\gvec(t))\Vert_{L^2(\Gamma)^3}^2
+\frac{1}{2\nu}\Vert\hvec(t)\Vert_{\Vdiv'}^2.
\end{align}
Once we have \eqref{inindir}, it is not difficult to argue as in the proof of Proposition \ref{dissprop}
(notice in particular that an estimate similar to \eqref{n6} can be obtained in this case by
exploiting once again the elliptic regularity estimates for $\dvec$ already used above).
In particular, it easy to check that, due to the assumption \eqref{g0ass} on $\gvec_0$, for
every $g\in\mathcal{H}_+(\gvec_0)$ we have
\begin{align}
&\Vert\gvec\Vert_{L^2_{tb}(0,\infty;H^{3/2}(\Gamma)^3)}
\leq\Vert\gvec_0\Vert_{L^2_{tb}(0,\infty;H^{3/2}(\Gamma)^3)},\quad
\Vert\partial_t\gvec\Vert_{L^2_{tb}(0,\infty;H^{-1/2}(\Gamma)^3)}
\leq\Vert\partial_t\gvec_0\Vert_{L^2_{tb}(0,\infty;H^{-1/2}(\Gamma)^3)}\nonumber\\
&\Vert\nabla_{\dvec}W_1(\gvec)\Vert_{L^2_{tb}(0,\infty;L^{2}(\Gamma)^3)}
\leq\Vert\nabla_{\dvec}W_1(\gvec_0)\Vert_{L^2_{tb}(0,\infty;L^{2}(\Gamma)^3)}.\nonumber
\end{align}
\color{black}
Therefore, we can first prove that $\mathcal{K}_{\gvec,\hvec}^+\subset\mathcal{W}_b^+$ and then recover
an inequality of the form \eqref{dissipativeest}, with $k'$ in place of $k$ and with $\Lambda_0$ depending on the constants $a_i$
and on the norms $\Vert\gvec_0\Vert_{L^2_{tb}(0,\infty;H^{3/2}(\Gamma)^3)}$,
$\Vert\partial_t\gvec_0\Vert_{L^2_{tb}(0,\infty;H^{-1/2}(\Gamma)^3)}$,
$\Vert\nabla_{\dvec}W_1(\gvec_0)\Vert_{L^2_{tb}(0,\infty;L^{2}(\Gamma)^3)}$ and
$\Vert\hvec_0\Vert_{L^2_{tb}(0,\infty;\Vdiv')}$.
\color{black}
If assumption {\bf (W1)}$^*$ and {\bf (W3)} hold in place of {\bf (W1)} we can argue as above and as in Proposition \ref{dissipativeest2}
and recover a dissipative estimate in the form \eqref{diss2}, where the constant $\Gamma_0$ now depends also
on the above norms of $\gvec_0$, $\partial_t\gvec_0$ and $\nabla_{\dvec}W_1(\gvec_0)$.\color{black}
We omit the details.

Let us now prove the closure property of the space of trajectories, assuming first, e.g., assumptions {\bf (W1)}
and {\bf (W2)}.
Let $\gvec_m\in\Xi_M$, $\hvec_m\in L^2(0,M;\Vdiv')$ and $\wvec_m\in\mathcal{K}_{\gvec_m,\hvec_m}^+$
with $\wvec_m:=[\uvec_m,\dvec_m]$ such that $\wvec_m\to\wvec$ in $\Theta_M$,
$\hvec_m\to\hvec$ in $L^2(0,M;\Vdiv')$ and $\gvec_m\to\gvec$ in $\Xi_M$. Then, we claim
that $\wvec\in\mathcal{K}_{\gvec,\hvec}^+$.
Indeed, we can argue as in the proof of Proposition \ref{closprop}. In particular,
in the energy inequality \eqref{eig}, written for each $\wvec_m$ and corresponding to $\gvec_m$
and $\hvec_m$, the second term on the right hand side can be estimated as
\begin{align*}
&\Big|\int_s^t\langle\partial_t\gvec_m,\partial_{\boldsymbol n}\dvec_m\rangle d\tau\Big|
\leq c\Vert\partial_t\gvec_m\Vert_{L^2(0,M;H^{-1/2}(\Gamma)^3)}\Vert\dvec_m\Vert_{L^2(0,M;H^2(\Omega)^3)}\leq c,
\end{align*}
for all $t\in[0,M]$ and a.a. $s\in[0,M]$ with $t\geq s$, due to the convergence assumption on the sequence
$\{\gvec_m\}$ and to \eqref{conv3}. Hence, using {\bf (W2)} and the convergence assumption on $\{\hvec_m\}$
we again infer that the right hand side of \eqref{eig} is bounded and recover the control
\eqref{H2cont}. In order to prove that $\wvec$ is a weak solution corresponding to $\gvec$ and $\hvec$
satisfying the energy inequality \eqref{eig}, we notice in particular that the convergence assumption
on $\{\gvec_m\}$ and \eqref{conv3} imply that $\dvec|_{\Gamma}=\gvec$
and furthermore that
$$ \int_s^t\langle\partial_t\gvec_m,\partial_{\boldsymbol n}\dvec_m\rangle d\tau
\to\int_s^t\langle\partial_t\gvec,\partial_{\boldsymbol n}\dvec\rangle d\tau.$$
Hence $\{\mathcal{K}_{\gvec,\hvec}^M\}_{\gvec\in\Xi_M,\hvec\in L^2(0,M;\Vdiv')}$
is $(\Theta_M,\Xi_M\times L^2(0,M;\Vdiv'))-$closed.

Furthermore, if the convergence assumption on $\{\hvec_m\}$ is replaced by
$\hvec_m\rightharpoonup\hvec$, weakly in $L^2(0,M;\Gdiv)$, then, arguing
as in \cite[Chap. XV, Prop. 1.1]{CVbook} (cf. end of the proof of Proposition \ref{closprop}),
we can still conclude that $\{\mathcal{K}_{\gvec,\hvec}^M\}_{\gvec\in\Xi_M,\hvec\in L^2_w(0,M;\Gdiv)}$
is $(\Theta_M,\Xi_M\times L^2_w(0,M;\Gdiv))-$closed.
\color{black}

Finally, if {\bf (W1)}$^*$ and {\bf (W3)} hold,  it is easy to check that $\{\mathcal{K}_{p,\gvec,\hvec}^M\}_{\gvec\in\Xi_M,\hvec\in L^2(0,M;\Vdiv')}$
is $(\Theta_{p,M},\Xi_M\times L^2(0,M;\Vdiv'))-$closed
and that $\{\mathcal{K}_{p,\gvec,\hvec}^M\}_{\gvec\in\Xi_M,\hvec\in L^2_w(0,M;\Gdiv)}$
is $(\Theta_{p,M},\Xi_M\times L^2_w(0,M;\Gdiv))-$closed, by the same argument as above and Proposition \ref{closprop2}.
This concludes the proof of Theorem~\ref{mainDir}
\color{black}.
\end{proof}

\begin{oss}{\upshape
If {\bf (W1)}$^*$ and {\bf (W3)} hold
the same conclusions of Theorem \ref{mainDir} still hold under a different assumption on
the Dirichlet datum $\gvec_0$. Indeed, suppose that $W_1$ satisfies the additional assumption
\begin{align}
&|\nabla_{\dvec}W_1(\dvec)|\leq C_3(1+|\dvec|^{p-1}),\qquad\forall\dvec\in\mathbb{R}^3,\qquad p>2,
\label{altassWbis}
\end{align}
(compare with {\bf (W3)}), and introduce the symbol space
$$\widetilde{\Xi}^+_{loc}:=\{\gvec\in L^{2p-2}_{loc}([0,\infty);H^{3/2}(\Gamma)^3):\gvec_t\in L^2_{loc}([0,\infty);H^{-1/2}(\Gamma)^3)\},$$
and also
\begin{align*}
&\widetilde{\Xi}_{loc,w}^+:=\{\gvec\in L^{2p-2}_{loc,w}([0,\infty);H^{3/2}(\Gamma)^3):\gvec_t\in L^2_{loc,w}([0,\infty);H^{-1/2}(\Gamma)^3)\}.
\end{align*}
Assume that
$$\gvec_0\in L^{2p-2}_{tb}(0,\infty;H^{3/2}(\Gamma)^3),\qquad\gvec_t\in L^2_{tb}(0,\infty;H^{-1/2}(\Gamma)^3).$$
Then, if we define the hull $\mathcal{H}_+(\gvec_0)$ as in \eqref{hullg}
with now the closure in $\widetilde{\Xi}_{loc,w}^+$, by arguing as in the proof above and using assumption \eqref{altassWbis}
in \eqref{p3}, we can still get a dissipative estimate in the form \eqref{diss2}
with the constant $\Gamma_0$ now depending on the norms
$\Vert\gvec_0\Vert_{L^{2p-2}_{tb}(0,\infty;H^{3/2}(\Gamma)^3)}$,
$\Vert\partial_t\gvec_0\Vert_{L^2_{tb}(0,\infty;H^{-1/2}(\Gamma)^3)}$
and $\Vert\nabla_{\dvec}W_1(\gvec_0)\Vert_{L^2_{tb}(0,\infty;L^{2}(\Gamma)^3)}$.
Finally, assuming in addition that $\partial_t\gvec_0$
is tr.-c. in $L^2_{loc}([0,\infty);H^{-1/2}(\Gamma)^3)$, then
the closure property of the space of trajectories can be recovered as well.
}
\end{oss}

\begin{oss}{\upshape
Notice that if $\wvec_m\to\wvec$ in $\Theta_M$, then due to the compact embedding
\eqref{AubLio}, we have
$$\partial_{\boldsymbol n}\dvec_m\to\partial_{\boldsymbol n}\dvec,\quad\mbox{strongly in }
L^2(0,M;H^{1/2-\delta}(\Gamma)^3),$$
for every $0<\delta<1/2$. Therefore, as far as the closure property
of the space of trajectories is concerned, the assumption
on $\partial_t\gvec_0$ in Theorem \ref{mainDir} could be replaced by
$$\partial_t\gvec_0\in L^2_{tb}(0,\infty;H^{-1/2+\delta}(\Gamma)^3),\quad 0<\delta<1/2.$$
}
\end{oss}

\color{black}


\section{Further properties of the trajectory attractor}
\label{furprop}

Let us consider only the case of homogeneous Neumann boundary condition for $\dvec$
and refer to both functional settings and assumptions on the potential introduced in
the previous section. The results of this section can be reproduced
also for the case of non-homogeneous Dirichlet boundary condition for $\dvec$ without any difficulty
(cf. Theorem \ref{mainDir}).

We start to discuss some structural properties of the trajectory attractor.

Denote by $Z(\hvec_0):=Z(\mathcal{H}_+(\hvec_0))$ the set of all complete symbols
in $\omega(\mathcal{H}_+(\hvec_0))$.
Recall that a function $\zetavec\in L_{loc}^2(\mathbb{R};\Vdiv')$
is a complete symbol in $\omega(\mathcal{H}_+(\hvec_0))$ if $\Pi_+T(t)\zetavec\in\omega(\mathcal{H}_+(\hvec_0))$
for all $t\in\mathbb{R}$, where $\Pi_+=\Pi_{[0,\infty)}$.
It can be proved (see \cite[Section 4]{CV} or \cite[Chap. XIV, Section 2]{CVbook}) that, due to the strict
invariance of $\omega(\mathcal{H}_+(\hvec_0))$, given a symbol $\hvec\in\omega(\mathcal{H}_+(\hvec_0))$
there exists at least one complete symbol $\widehat{\hvec}$ (not necessarily unique) which is an extension
of $\hvec$ on $(-\infty,0]$ and such that $\Pi_+T(t)\widehat{\hvec}\in\omega(\mathcal{H}_+(\hvec_0))$
for all $t\in\mathbb{R}$. Note that we have $\Pi_+Z(\hvec_0)=\omega(\mathcal{H}_+(\hvec_0))$.

Let us refer first to the functional setting introduced in Theorem \ref{main}.

To every complete symbol $\zetavec\in Z(\hvec_0)$ there corresponds by \cite[Chap. XIV, Definition 2.5]{CVbook}
(see also \cite[Definition 4.4]{CV}) the kernel $\mathcal{K}_{\zetavec}$ in $\mathcal{W}_b$
which consists of the union of all complete trajectories which belong to $\mathcal{W}_b$,
i.e., all weak solutions $\wvec=[\uvec,\dvec]:\mathbb{R}\to \Gdiv\times\V$
with external force $\zetavec\in Z(\hvec_0)$ (in the sense of Definition \ref{wsdefn} with
the interval $[0,\infty)$ replaced by $\mathbb{R}$)
satisfying \eqref{eist} on $\mathbb{R}$ (i.e., for all $t\geq s$ and for a.a. $s\in\mathbb{R}$)
that belong to $\mathcal{W}_b$. We recall that the space $(\mathcal{W}_b,\rho_{\mathcal{W}_b})$ is defined
as the space $(\mathcal{W}_b^+,\rho_{\mathcal{W}_b}^+)$ (see \eqref{Fb} and \eqref{dFb}) with the time
interval $(0,\infty)$ replaced by $\mathbb{R}$. The space $(\mathcal{W}_{loc},\Theta_{loc})$ can be defined in the same way.

Set
$$\mathcal{K}_{Z(\hvec_0)}:=\bigcup_{\zetavec\in Z(\hvec_0)}\mathcal{K}_{\zetavec}.$$
Then, if the assumptions of Theorem \ref{main} hold
we also have (see, e.g., \cite[Theorem 4.1]{CV})
$$\mathcal{A}_{\mathcal{H}_+(\hvec_0)}=\mathcal{A}_{\omega(\mathcal{H}_+(\hvec_0))}=\Pi_+\mathcal{K}_{Z(\hvec_0)},
$$
and the set $\mathcal{K}_{Z(\hvec_0)}$ is compact in $\Theta_{loc}$ and bounded in $\mathcal{W}_b$.

On the other hand, it is not difficult to see that, under the assumptions of Theorem \ref{main},
$\mathcal{K}_{\zetavec}\neq\emptyset$ for all $\zetavec\in Z(\hvec_0)$.
Indeed, by virtue of \cite[Theorem 4.1]{CV} (see also \cite[Chap. XIV, Theorem 2.1]{CVbook}),
this is a consequence of the fact that the family $\{\mathcal{K}_{\hvec}^+\}_{\hvec\in\mathcal{H}_+(\hvec_0)}$
of trajectory spaces satisfies the following condition: there exists $R>0$ such that
$B_{\mathcal{W}_b^+}(0,R)\cap\mathcal{K}_{\hvec}^+\neq\emptyset$ for all $\hvec\in\mathcal{H}_+(\hvec_0)$.
In order to check this condition fix an initial datum $\wvec_0^{\ast}=[\uvec_0^{\ast},\dvec_0^{\ast}]$,
with $\uvec_0^{\ast},\dvec_0^{\ast}$ taken as in Theorem \ref{existence}. We know that for every
$\hvec\in\mathcal{H}_+(\hvec_0)$ there exists a trajectory $\wvec_{\hvec}^{\ast}\in\mathcal{K}_{\hvec}^+$ such that
$\wvec_{\hvec}^{\ast}(0)=\wvec_0^{\ast}$ and such that the energy inequality \eqref{eist} holds
for all $t\geq s$ and for a.a. $s\in(0,\infty)$, {\itshape including $s=0$}. Arguing as in Proposition
\ref{dissprop} (cf. \eqref{n0} written for $s=0$ and all $t\geq 0$) we get an estimate
of the form $\rho_{\mathcal{W}_b^+}(\wvec_{\hvec}^{\ast},0)\leq \Lambda(\wvec_0^{\ast},\hvec_0)$ (see also \eqref{hh0}),
where the positive constant $\Lambda$ depends on $\mathcal{E}(\wvec_0^{\ast})$ and on the norm
$\|\hvec_0\|_{L^2_{tb}(0,\infty;\Vdiv')}$. The above condition is thus fulfilled by choosing $R=\Lambda(\wvec_0^{\ast},\hvec_0)$.

In the case we consider the functional setting of Theorem \ref{main2},
we can similarly introduce the kernel $\mathcal{K}_{p,\zetavec}$ in $\mathcal{W}_{p,b}$,
with the Banach space $(\mathcal{W}_{p,b},\rho_{\mathcal{W}_{p,b}})$ always defined as the space
$(\mathcal{W}_{p,b}^+,\rho_{\mathcal{W}_{p,b}^+})$ (see \eqref{Fpb} and \eqref{Fpbnorm}) with the time interval $(0,\infty)$ replaced
by $\mathbb{R}$. The space $(\mathcal{W}_{p,loc},\Theta_{p,loc})$ can be defined in the same way
from the space $(\mathcal{W}_{p,loc}^+,\Theta_{p,loc}^+)$.
Hence, in this case, if the assumptions of Theorem \ref{main2} hold, we have
$$\mathcal{A}_{p,\mathcal{H}_+(\hvec_0)}=\mathcal{A}_{p,\omega(\mathcal{H}_+(\hvec_0))}
=\Pi_+\mathcal{K}_{p,Z(\hvec_0)}:=\Pi_+\bigcup_{\zetavec\in Z(\hvec_0)}\mathcal{K}_{p,\zetavec},$$
and the set $\mathcal{K}_{p,Z(\hvec_0)}$ is compact in $\Theta_{p,loc}$ and bounded in $\mathcal{W}_{p,b}$.
The proof that $\mathcal{K}_{p,\zetavec}\neq\emptyset$ for all $\zetavec\in Z(\hvec_0)$
is exactly the same as above.

As far as the attraction properties are concerned, we observe that, due to compactness results,
the trajectory attractor attracts the subsets
of the family $\mathcal{B}_{\mathcal{H}_+(\hvec_0)}^+$
(if we refer to Theorem \ref{main}), or the subsets
of $\mathcal{K}_{p,\mathcal{H}_+(\hvec_0)}^+$
bounded in $\mathcal{W}_{p,b}^+$ (if we refer to Theorem \ref{main2}),
in some strong topologies. Indeed, set
\begin{align}
&\mathbb{X}_{\delta_1,\delta_2}:=H^{\delta_1}(\Omega)^3\times H^{1+\delta_2}(\Omega)^3,\qquad
\mathbb{Y}_{\delta_1,\delta_2}:=H^{-\delta_1}(\Omega)^3\times H^{\delta_2}(\Omega)^3,\\
&\mathbb{Y}_{\delta_1,\delta_2}^s:=H^{-\delta_1}(\Omega)^3\times (H^{\delta_2}(\Omega)^3\cap L^s(\Omega)^3)
\end{align}
where $0\leq\delta_1,\delta_2<1$ and $2\leq s<p$. Then, by using the compact embeddings
\begin{align}
&
L^2(0,M;\Vdiv\times H^2(\Omega)^3)\cap H^1(0,M;W^{-1,3/2}(\Omega)^3\times L^{3/2}(\Omega)^3)\hookrightarrow\hookrightarrow
 L^2(0,M;\mathbb{X}_{\delta_1,\delta_2}),\nonumber\\
&L^{\infty}(0,M;\Gdiv\times\V)\cap H^1(0,M;W^{-1,3/2}(\Omega)^3\times L^{3/2}(\Omega)^3)\hookrightarrow\hookrightarrow
C([0,M];\mathbb{Y}_{\delta_1,\delta_2}),\nonumber\\
\nonumber
&L^{\infty}(0,M;\Gdiv\times(\V\cap L^p(\Omega)^3))\cap H^1(0,M;W^{-1,3/2}(\Omega)^3\times L^{3/2}(\Omega)^3)\\
\nonumber
&\qquad\qquad\qquad\hookrightarrow\hookrightarrow
C([0,M];\mathbb{Y}_{\delta_1,\delta_2}^s),\nonumber
\end{align}
then Theorem \ref{main} and Theorem \ref{main2}
imply the following two corollaries (see \cite[Chap. XIV, Theorem 2.2]{CVbook})
\begin{cor}
\label{statt1}
Let {\bf (W1)} and {\bf (W2)} hold and assume $\hvec_0\in L^2_{tb}(0,\infty;\Vdiv')$. Then, for every $0\leq\delta_1,\delta_2<1$
the trajectory attractor $\mathcal{A}_{\mathcal{H}_+(\hvec_0)}$ from Theorem \ref{main}
is compact in $L^2_{loc}([0,\infty);\mathbb{X}_{\delta_1,\delta_2})\cap C([0,\infty);\mathbb{Y}_{\delta_1,\delta_2})$,
bounded in $L^2_{tb}(0,\infty;\mathbb{X}_{\delta_1,\delta_2})\cap C_b([0,\infty);\mathbb{Y}_{\delta_1,\delta_2})$,
and for every
$B\in\mathcal{B}_{\mathcal{H}_+(\hvec_0)}^+$ and
every $M>0$ we have, for $t\to+\infty$
\begin{align}
&\mbox{dist}_{L^2(0,M;\mathbb{X}_{\delta_1,\delta_2})}
\Big(\Pi_{[0,M]}T(t)B,\Pi_{[0,M]}\mathcal{A}_{\mathcal{H}_+(\hvec_0)}\Big)\to 0,\nonumber\\
&\mbox{dist}_{C([0,M];\mathbb{Y}_{\delta_1,\delta_2})}
\Big(\Pi_{[0,M]}T(t)B,\Pi_{[0,M]}\mathcal{A}_{\mathcal{H}_+(\hvec_0)}\Big)\to 0.\nonumber
\end{align}
\end{cor}

\begin{cor}
\label{statt2}
Let {\bf (W1)}$^*$ and {\bf (W3)} hold and assume $\hvec_0\in L^2_{tb}(0,\infty;\Vdiv')$. Then, for every $0\leq\delta_1,\delta_2<1$
and every $2\leq s<p$ the trajectory attractor $\mathcal{A}_{p,\mathcal{H}_+(\hvec_0)}$ from Theorem \ref{main2}
is compact in $L^2_{loc}([0,\infty);\mathbb{X}_{\delta_1,\delta_2})\cap C([0,\infty);\mathbb{Y}_{\delta_1,\delta_2}^s)$,
bounded in $L^2_{tb}(0,\infty;\mathbb{X}_{\delta_1,\delta_2})\cap C_b([0,\infty);\mathbb{Y}_{\delta_1,\delta_2}^s)$,
and for every
$B\subset\mathcal{K}_{p,\mathcal{H}_+(\hvec_0)}^+$ bounded in $\mathcal{W}_{p,b}^+$
and
every $M>0$ we have, for $t\to+\infty$
\begin{align}
&\mbox{dist}_{L^2(0,M;\mathbb{X}_{\delta_1,\delta_2})}
\Big(\Pi_{[0,M]}T(t)B,\Pi_{[0,M]}\mathcal{A}_{p,\mathcal{H}_+(\hvec_0)}\Big)\to 0,\nonumber\\
&\mbox{dist}_{C([0,M];\mathbb{Y}_{\delta_1,\delta_2}^s)}
\Big(\Pi_{[0,M]}T(t)B,\Pi_{[0,M]}\mathcal{A}_{p,\mathcal{H}_+(\hvec_0)}\Big)\to 0.\nonumber
\end{align}
\end{cor}
In Corollary \ref{statt1} and Corollary \ref{statt2} we have denoted by
$\mbox{dist}_{X}(A,B)$ the Hausdorff semidistance in the
Banach space $X$ between $A,B\subset X$.

Let us now define, for every $B\subset\mathcal{K}_{\mathcal{H}_+(\hvec_0)}^+$
and every $B_p\subset\mathcal{K}_{p,\mathcal{H}_+(\hvec_0)}^+$ , the sections
\begin{align}
& B(t):=\Big\{[\uvec(t),\dvec(t)]:[\uvec,\dvec]\in B \Big\}\subset\mathbb{Y}_{\delta_1,\delta_2},
\qquad t\geq 0,\nonumber\\
& B_p(t):=\Big\{[\uvec(t),\dvec(t)]:[\uvec,\dvec]\in B_p \Big\}\subset\mathbb{Y}_{\delta_1,\delta_2}^s,
\qquad t\geq 0.
\end{align}
Similarly we set
\begin{align}
&
\mathcal{A}_{\mathcal{H}_+(\hvec_0)}(t):=\Big\{[\uvec(t),\dvec(t)]:[\uvec,\dvec]
\in\mathcal{A}_{\mathcal{H}_+(\hvec_0)} \Big\}\subset
\mathbb{Y}_{\delta_1,\delta_2},
\qquad t\geq 0,\nonumber\\
&
\mathcal{K}_{Z(\hvec_0)}(t):=\Big\{[\uvec(t),\dvec(t)]:[\uvec,\dvec]\in \mathcal{K}_{Z(\hvec_0)} \Big\}\subset
\mathbb{Y}_{\delta_1,\delta_2},
\qquad t\in\mathbb{R},\nonumber\\
&
\mathcal{A}_{p,\mathcal{H}_+(\hvec_0)}(t):=\Big\{[\uvec(t),\dvec(t)]:[\uvec,\dvec]
\in\mathcal{A}_{p,\mathcal{H}_+(\hvec_0)} \Big\}\subset
\mathbb{Y}_{\delta_1,\delta_2}^s,
\qquad t\geq 0,\nonumber\\
&
\mathcal{K}_{p,Z(\hvec_0)}(t):=\Big\{[\uvec(t),\dvec(t)]:[\uvec,\dvec]\in \mathcal{K}_{p,Z(\hvec_0)} \Big\}\subset
\mathbb{Y}_{\delta_1,\delta_2}^s,
\qquad t\in\mathbb{R}.\nonumber
\end{align}
Then, as a further consequence of Theorem \ref{main} and Theorem \ref{main2} we have
(see \cite[Chap. XIV, Definition 2.6, Corollary 2.2]{CVbook}) the following two corollaries
\begin{cor}
Let {\bf (W1)} and {\bf (W2)} hold and assume that
$\hvec_0$ is tr.-c. in $L^2_{loc}([0,\infty);\Vdiv')$ or
$\hvec_0\in L^2_{tb}(0,\infty;\Gdiv)$.
 Then the bounded subset
$$\mathcal{A}_{gl}:=\mathcal{A}_{\mathcal{H}_+(\hvec_0)}(0)=\mathcal{K}_{Z(\hvec_0)}(0)$$
is the uniform (w.r.t. $\hvec\in\mathcal{H}_+(\hvec_0)$) global attractor in $\mathbb{Y}_{\delta_1,\delta_2}$,
$0\leq\delta_1,\delta_2< 1$, of system \eqref{sy1}--\eqref{sy5}, namely
(i) $\mathcal{A}_{gl}$ is compact in $\mathbb{Y}_{\delta_1,\delta_2}$,
(ii) $\mathcal{A}_{gl}$ satisfies the attracting property
$$\mbox{dist}_{\mathbb{Y}_{\delta_1,\delta_2}}(B(t),\mathcal{A}_{gl})\to 0,
\qquad t\to+\infty,$$
for every $B\in\mathcal{B}_{\mathcal{H}_+(\hvec_0)}^+$, and (iii) $\mathcal{A}_{gl}$
is the minimal set satisfying (i) and (ii).
\end{cor}
\begin{cor}
Let {\bf (W1)$^*$} and {\bf (W3)} hold and assume
that $\hvec_0$ is tr.-c. in $L^2_{loc}([0,\infty);\Vdiv')$ or
$\hvec_0\in L^2_{tb}(0,\infty;\Gdiv)$.
 Then the bounded subset
$$\mathcal{A}_{p,gl}:=\mathcal{A}_{p,\mathcal{H}_+(\hvec_0)}(0)=\mathcal{K}_{p,Z(\hvec_0)}(0)$$
is the uniform (w.r.t. $\hvec\in\mathcal{H}_+(\hvec_0)$) global attractor in $\mathbb{Y}_{\delta_1,\delta_2}^s$,
$0\leq\delta_1,\delta_2< 1$ and $2\leq s<p$, of system \eqref{sy1}--\eqref{sy5}, namely
(i) $\mathcal{A}_{p,gl}$ is compact in $\mathbb{Y}_{\delta_1,\delta_2}^s$,
(ii) $\mathcal{A}_{p,gl}$ satisfies the attracting property
$$\mbox{dist}_{\mathbb{Y}_{\delta_1,\delta_2}^s}(B_p(t),\mathcal{A}_{p,gl})\to 0,
\qquad t\to+\infty,$$
for every $B_p\subset\mathcal{K}_{p,\mathcal{H}_+(\hvec_0)}^+$ bounded in $\mathcal{W}_{p,b}^+$, and (iii) $\mathcal{A}_{p,gl}$
is the minimal set satisfying (i) and (ii).
\end{cor}

\end{document}